\documentclass[a4paper,11pt]{amsart}

\usepackage[latin1]{inputenc}
\usepackage[cyr]{aeguill}
\usepackage[english]{babel}
\usepackage{tikz}
\usepackage[all]{xy}

\usepackage[T1]{fontenc}
\usepackage{amsmath}
\usepackage{amsfonts}
\usepackage{amssymb}
\usepackage{amsthm,eepic}
\usepackage{mathrsfs}
\usepackage{enumitem}
\usepackage{graphicx}
\usepackage{footnote}
\usepackage{hyperref}
\hypersetup{colorlinks=true,    linkcolor=blue,
citecolor=red,       filecolor=BrickRed,       urlcolor=darkgreen
}
\usepackage[normalem]{ulem}
\usepackage{cancel}

\renewcommand{\geq}{\geqslant}
\renewcommand{\leq}{\leqslant}
\newtheorem{thm}{Theorem}[section]

\newtheorem{rem}[thm]{Remark}

\newtheorem{prop}[thm]{Proposition}
\newtheorem{lem}[thm]{Lemma}
\newtheorem{ex}[thm]{Example}

\newtheorem*{thmintro}{Theorem}
\newtheorem*{conj}{Conjecture}

\def\Z{\mathbb{Z}}
\def\N{\mathbb{N}}
\def\C{\mathbb{C}}

\def\Q{\mathbb{Q}}
\def\P1{\mathbb{P}^{1}}
\def\Etproj{\overline{E}}
\def\iup{{\widetilde{\iota}}}
\def\rx{r_{x}}
\def\ry{r_{y}}

\usepackage{color}
\definecolor{darkgreen}{rgb}{0,0.4,0}
\definecolor{MyDarkBlue}{rgb}{0,0.08,0.50}
\definecolor{BrickRed}{rgb}{0.65,0.08,0}

\definecolor{dblue}  {RGB}{20,66,129}
\definecolor{nred}   {RGB}{224,0,0}
\definecolor{dgreen} {RGB}{78,138,21}
\definecolor{Apricot} {RGB}{255, 170, 123} 
\definecolor{dpurple}  {RGB}{53,21,93}

\makeatletter
\def\testb#1{\testb@i#1,,\@nil}%
\def\testb@i#1,#2,#3\@nil{%
  \draw[->, thick] (O) --++(#1);
  \ifx\relax#2\relax\else\testb@i#2,#3\@nil\fi}
\makeatother   
\newcommand{\makediag}[1]{
    \coordinate (O) at (0,0); \coordinate (N) at (0,0.8);
    \coordinate (NE) at (0.8,0.8); \coordinate (E) at (0.8,0);
    \coordinate (SE) at (0.8,-0.8); \coordinate (S) at (0,-0.8);
    \coordinate (SW) at (-0.8,-0.8);\coordinate (W) at (-0.8,0);
    \coordinate (N2E) at (1.6,0.8);\coordinate (S2W) at (-1.6,-0.8);
    \coordinate (NW) at (-0.8,0.8); \coordinate (B1) at (1.2,1.2);
    \coordinate (B2) at (-1.2,-1.2);
    \testb{#1}
} 
\newcommand{\diagr}[1]{
  \begin{tikzpicture}[scale=0.8]\makediag{#1}\end{tikzpicture}
}

\title{On the nature of four models of symmetric walks avoiding a quadrant}

\author{Thomas Dreyfus}
\address{Institut de Recherche Math\'ematique Avanc\'ee, U.M.R. 7501 Universit\'e de Strasbourg et C.N.R.S. 7, rue Ren\'e Descartes 67084 Strasbourg, France}
\email{dreyfus@math.unistra.fr}
\author{Am\'elie Trotignon}
\address{Institute for Algebra, Johannes Kepler University, Altenbergerstrasse 69 4040 Linz, Austria}\email{amelie.trotignon@jku.at}

\date{\today}

\begin{document}

\thanks{This project has received funding from the European Research Council (ERC) under the European Union's Horizon 2020 research and innovation programme under the Grant Agreement No 759702. This project has received funding from the ANR de rerum natura ANR-19-CE40-0018. A. Trotignon was supported by the Austrian Science Fund (FWF) grant FWF05004.}

\begin{abstract}
We study the nature of the generating series of some models of walks with small steps in the three quarter plane. More precisely, we restrict ourselves to the situation where the group is infinite, the kernel has genus one, and the step set is diagonally symmetric (i.e., with no steps in anti-diagonal directions). In that situation, after a transformation of the plane, we derive a quadrant-like functional equation. Among the four models of walks, we obtain, using difference Galois theory, that three of them have a differentially transcendental  generating series, and one has a differentially algebraic generating series.
\end{abstract}

\setcounter{tocdepth}{1}
\maketitle 
\tableofcontents
\section*{Introduction}

In this paper, we consider walks with small steps in the three quarter plane
\begin{equation*}
     \mathcal C=\{(i,j)\in\mathbb Z^2: i\geq0 \text{ or } j\geq0\}.   
\end{equation*}

More precisely, we encode the eight cardinal directions of the plane by pairs of integers $(i,j)$ with $i,j \in \{0,\pm 1\}$.  We consider models of walks in the three quarter plane $\mathcal C$ satisfying the following properties: 
\begin{itemize} 
\item the walks start at $(0,0)$; 
\item the small steps walks take steps from $\mathcal{S}\subseteq \{0,\pm 1\}^{2}$. This is called the step set. Unless explicitly mentioned, we suppose in this paper that all models are small step models.
\end{itemize}

We introduce the corresponding trivariate generating series
$$
C(x,y;t):=\displaystyle \sum_{n=0}^{\infty}\sum_{(i,j)\in \mathcal{C}}c_{i,j}(n)x^{i}y^{j}t^{n},
$$
where $ c_{i,j}(n)$ denotes the number of walks in $\mathcal{C}$ reaching  the position $(i,j)$ from the initial position $(0,0)$ after $n$ steps in $\mathcal{S}$. We stress the fact that the paper is written in a more general weighted context, see Section \ref{sec:Kernel-FctEq}, but to simplify the exposition, we are going to restrict ourselves to the unweighted case in the introduction.

\vspace*{0.5cm}

\paragraph{\textbf{Context.}} Enumeration of lattice walks is a central question in combinatorics. Most studies have been done on walks confined to convex cones with various methods and techniques which enrich each other in many ways:  combinatorics~\cite{MiRe-09, BMM,  MeMi-14, MeMi-16}, complex analysis~\cite{FIM, KuRa-11, RaschelJEMS, FaRa-12, kurkova2012functions, KuRa-15}, probability theory~\cite{DeWa-15}, computer algebra~\cite{BostanKauersTheCompleteGenerating, BoChVHKaPe-17}, and Galois theory of difference equations~\cite{DHRS, dreyfus2020walks, dreyfus2017differential, dreyfus2019length}. Three natural topics arise in walk studies:  closed-form expressions for the generating series of the number of walks, the asymptotic behavior of the number of excursions, and the nature of the trivariate generating series. In this article, we are interested in the nature of the generating series $C(x,y;t)$ of walks avoiding a quadrant\footnote{Thereafter, the expressions {\it avoiding a quadrant}, {\it confined to the three quadrants}, and {\it confined to the three quarter plane}, will be used with no distinction.}.  A function can be rational, algebraic, D-finite, D-algebraic or D-transcendental with the following hierarchical chain
\begin{equation*}
\text{rational} \quad \varsubsetneq \quad \text{algebraic} \quad \varsubsetneq \quad \text{D-finite}\quad  \varsubsetneq \quad \text{D-algebraic}. 
\end{equation*} 
Rational and algebraic functions are classical notions. 
By $C(x,y;t)$ D-finite (resp.~D-algebraic) we mean that all  $x\mapsto C(x,y;t)$, $y\mapsto C(x,y;t)$, $t\mapsto C(x,y;t)$ satisfy a nontrivial linear (resp.~algebraic) differential equation with coefficients in $\Q(x,y,t)$. We say that $C(x,y;t)$ is D-transcendental if it is not D-algebraic. We refer to Section \ref{sec:DiffTranscendance} for more details. \par 
Walks in the plane always have rational generating series, while walks in the half plane always have algebraic generating series, see~\cite{BaFl-02}. The next step is to consider walks in the quarter plane. The situation gets much more complicated, and the nature of the generating series depends on the choice of the step set $\mathcal{S}$.
This question has generated a great interest and the determination of the nature of the generating series is now complete, see Figure~\ref{figcas}. More precisely, the  study has been started with the seminal paper \cite{BMM}. In this article, the authors prove that after elimination of trivial or one dimensional cases, and after considering symmetries, only $79$ of the original $2^{8}=256$ possible step sets $\mathcal{S}$ remain to study. They introduced a notion of group associated to the step set, and using combinatorial methods, proved that amongst the $23$ step sets with finite group, $22$ have a generating series which is D-finite, and even algebraic in $3$ cases. With computer algebra, the last finite group  case was solved  in \cite{BostanKauersTheCompleteGenerating} (see also \cite{fayolle2010holonomy,melczer2015asymptotics}), and proved to be algebraic. The study of the $56$ step sets with infinite group is more tricky. An algebraic curve is associated with a step set which has genus one in $51$ cases and genus zero in the $5$ other cases. With complex analysis, the authors of \cite{kurkova2012functions} proved that in the genus one case with infinite group, the generating series is not D-finite. Their work is based on the study of the uniformization of an elliptic curve initiated in \cite{FIM}. The classification between D-algebraic and D-transcendental is more recent.  In \cite{bernardi2017counting}, it is proved that in the $51$ above cited step sets, $9$ have a generating series which is D-algebraic. The difference  Galois theory, see \cite{DHRS,dreyfus2020walks,dreyfus2019length}, allows one to prove that the other $47$ step sets lead to D-transcendental generating series. This completes the study of the nature of the generating series of  walks in the quarter plane case. 
We refer to \cite{bostan2018counting} for a starting point of the study of walks with large steps (that is walks with arbitrary step set $\mathcal{S}  \subseteq \Z^{2}$), and for instance to \cite{dreyfus2017differential}, for generalization of some results in the weighted context.

\begin{figure}[!h]
\begin{trivlist}
\item  \begin{center}
Algebraic cases
\end{center}
$\begin{array}{llll}
\begin{tikzpicture}[scale=.2, baseline=(current bounding box.center)]
\draw[thick,->](0,0)--(1,1);
\draw[thick,->](0,0)--(-1,0);
\draw[thick,->](0,0)--(0,-1);
\end{tikzpicture}
& 
\begin{tikzpicture}[scale=.2, baseline=(current bounding box.center)]

\draw[thick,->](0,0)--(0,1);
\draw[thick,->](0,0)--(1,0);
\draw[thick,->](0,0)--(-1,-1);
\end{tikzpicture} 
&
\begin{tikzpicture}[scale=.2, baseline=(current bounding box.center)]

\draw[thick,->](0,0)--(0,1);
\draw[thick,->](0,0)--(1,1);
\draw[thick,->](0,0)--(-1,0);
\draw[thick,->](0,0)--(1,0);
\draw[thick,->](0,0)--(-1,-1);
\draw[thick,->](0,0)--(0,-1);
\end{tikzpicture}
&
\begin{tikzpicture}[scale=.2, baseline=(current bounding box.center)]

\draw[thick,->](0,0)--(1,1);
\draw[thick,->](0,0)--(-1,-1);
\draw[thick,->](0,0)--(1,0);
\draw[thick,->](0,0)--(-1,0);
\end{tikzpicture}
\end{array}$

\item \begin{center} D-finite cases \end{center}

$\begin{array}{llllllllllllllll}
\begin{tikzpicture}[scale=.2, baseline=(current bounding box.center)]

\draw[thick,->](0,0)--(0,1);
\draw[thick,->](0,0)--(-1,0);
\draw[thick,->](0,0)--(1,0);
\draw[thick,->](0,0)--(0,-1);
\end{tikzpicture}
&
\begin{tikzpicture}[scale=.2, baseline=(current bounding box.center)]

\draw[thick,->](0,0)--(-1,1);
\draw[thick,->](0,0)--(1,1);
\draw[thick,->](0,0)--(-1,-1);
\draw[thick,->](0,0)--(1,-1);
\end{tikzpicture}
&
\begin{tikzpicture}[scale=.2, baseline=(current bounding box.center)]

\draw[thick,->](0,0)--(-1,1);
\draw[thick,->](0,0)--(0,1);
\draw[thick,->](0,0)--(1,1);
\draw[thick,->](0,0)--(-1,-1);
\draw[thick,->](0,0)--(0,-1);
\draw[thick,->](0,0)--(1,-1);
\end{tikzpicture}
&
\begin{tikzpicture}[scale=.2, baseline=(current bounding box.center)]

\draw[thick,->](0,0)--(-1,1);
\draw[thick,->](0,0)--(0,1);
\draw[thick,->](0,0)--(1,1);
\draw[thick,->](0,0)--(-1,0);
\draw[thick,->](0,0)--(1,0);
\draw[thick,->](0,0)--(-1,-1);
\draw[thick,->](0,0)--(0,-1);
\draw[thick,->](0,0)--(1,-1);
\end{tikzpicture}
&
\begin{tikzpicture}[scale=.2, baseline=(current bounding box.center)]

\draw[thick,->](0,0)--(-1,1);
\draw[thick,->](0,0)--(1,1);
\draw[thick,->](0,0)--(0,-1);
\end{tikzpicture}
&
\begin{tikzpicture}[scale=.2, baseline=(current bounding box.center)]

\draw[thick,->](0,0)--(-1,1);
\draw[thick,->](0,0)--(1,1);
\draw[thick,->](0,0)--(-1,0);
\draw[thick,->](0,0)--(1,0);
\draw[thick,->](0,0)--(0,-1);
\end{tikzpicture}
&
\begin{tikzpicture}[scale=.2, baseline=(current bounding box.center)]

\draw[thick,->](0,0)--(-1,1);
\draw[thick,->](0,0)--(0,1);
\draw[thick,->](0,0)--(1,1);
\draw[thick,->](0,0)--(0,-1);
\end{tikzpicture}
&
\begin{tikzpicture}[scale=.2, baseline=(current bounding box.center)]

\draw[thick,->](0,0)--(-1,1);
\draw[thick,->](0,0)--(0,1);
\draw[thick,->](0,0)--(1,1);
\draw[thick,->](0,0)--(-1,0);
\draw[thick,->](0,0)--(1,0);
\draw[thick,->](0,0)--(0,-1);
\end{tikzpicture}
&
\begin{tikzpicture}[scale=.2, baseline=(current bounding box.center)]

\draw[thick,->](0,0)--(-1,1);
\draw[thick,->](0,0)--(0,1);
\draw[thick,->](0,0)--(1,1);
\draw[thick,->](0,0)--(-1,-1);
\draw[thick,->](0,0)--(1,-1);
\end{tikzpicture}
& 
\begin{tikzpicture}[scale=.2, baseline=(current bounding box.center)]

\draw[thick,->](0,0)--(-1,1);
\draw[thick,->](0,0)--(0,1);
\draw[thick,->](0,0)--(1,1);
\draw[thick,->](0,0)--(-1,0);
\draw[thick,->](0,0)--(1,0);
\draw[thick,->](0,0)--(-1,-1);
\draw[thick,->](0,0)--(1,-1);
\end{tikzpicture}
&
\begin{tikzpicture}[scale=.2, baseline=(current bounding box.center)]

\draw[thick,->](0,0)--(0,1);
\draw[thick,->](0,0)--(-1,-1);
\draw[thick,->](0,0)--(0,-1);
\draw[thick,->](0,0)--(1,-1);
\end{tikzpicture}
&
\begin{tikzpicture}[scale=.2, baseline=(current bounding box.center)]

\draw[thick,->](0,0)--(0,1);
\draw[thick,->](0,0)--(-1,0);
\draw[thick,->](0,0)--(1,0);
\draw[thick,->](0,0)--(-1,-1);red
\draw[thick,->](0,0)--(0,-1);
\draw[thick,->](0,0)--(1,-1);
\end{tikzpicture}
&
\begin{tikzpicture}[scale=.2, baseline=(current bounding box.center)]

\draw[thick,->](0,0)--(-1,1);
\draw[thick,->](0,0)--(1,1);
\draw[thick,->](0,0)--(-1,-1);
\draw[thick,->](0,0)--(0,-1);
\draw[thick,->](0,0)--(1,-1);
\end{tikzpicture}
&
\begin{tikzpicture}[scale=.2, baseline=(current bounding box.center)]

\draw[thick,->](0,0)--(-1,1);
\draw[thick,->](0,0)--(1,1);
\draw[thick,->](0,0)--(-1,0);
\draw[thick,->](0,0)--(1,0);
\draw[thick,->](0,0)--(-1,-1);
\draw[thick,->](0,0)--(0,-1);
\draw[thick,->](0,0)--(1,-1);
\end{tikzpicture}
&
\begin{tikzpicture}[scale=.2, baseline=(current bounding box.center)]

\draw[thick,->](0,0)--(0,1);
\draw[thick,->](0,0)--(1,-1);
\draw[thick,->](0,0)--(-1,-1);
\end{tikzpicture} 
&
\begin{tikzpicture}[scale=.2, baseline=(current bounding box.center)]

\draw[thick,->](0,0)--(0,1);
\draw[thick,->](0,0)--(1,-1);
\draw[thick,->](0,0)--(-1,-1);
\draw[thick,->](0,0)--(1,0);
\draw[thick,->](0,0)--(-1,0);
\end{tikzpicture}\\
\begin{tikzpicture}[scale=.2, baseline=(current bounding box.center)]

\draw[thick,->](0,0)--(0,1);
\draw[thick,->](0,0)--(-1,0);
\draw[thick,->](0,0)--(1,-1);
\end{tikzpicture}
&
\begin{tikzpicture}[scale=.2, baseline=(current bounding box.center)]

\draw[thick,->](0,0)--(-1,1);
\draw[thick,->](0,0)--(0,1);
\draw[thick,->](0,0)--(-1,0);
\draw[thick,->](0,0)--(1,0);
\draw[thick,->](0,0)--(0,-1);
\draw[thick,->](0,0)--(1,-1);
\end{tikzpicture}&
\begin{tikzpicture}[scale=.2, baseline=(current bounding box.center)]

\draw[thick,->](0,0)--(-1,1);
\draw[thick,->](0,0)--(1,-1);
\draw[thick,->](0,0)--(1,0);
\draw[thick,->](0,0)--(-1,0);
\end{tikzpicture}
&&&&&&&&&&&&
\end{array}$
\item
\begin{center}  D-algebraic cases \end{center}
$\begin{array}{lllllllll}
\begin{tikzpicture}[scale=.2, baseline=(current bounding box.center)]

\draw[thick,->](0,0)--(0,1);
\draw[thick,->](0,0)--(1,0);
\draw[thick,->](0,0)--(-1,-1);
\draw[thick,->](0,0)--(0,-1);
\end{tikzpicture}
&
\begin{tikzpicture}[scale=.2, baseline=(current bounding box.center)]

\draw[thick,->](0,0)--(0,1);
\draw[thick,->](0,0)--(1,0);
\draw[thick,->](0,0)--(-1,-1);
\draw[thick,->](0,0)--(1,-1);
\end{tikzpicture}
&
\begin{tikzpicture}[scale=.2, baseline=(current bounding box.center)]

\draw[thick,->](0,0)--(0,1);
\draw[thick,->](0,0)--(-1,0);
\draw[thick,->](0,0)--(1,0);
\draw[thick,->](0,0)--(1,-1);
\end{tikzpicture}
&
\begin{tikzpicture}[scale=.2, baseline=(current bounding box.center)]

\draw[thick,->](0,0)--(0,1);
\draw[thick,->](0,0)--(1,0);
\draw[thick,->](0,0)--(-1,-1);
\draw[thick,->](0,0)--(0,-1);
\draw[thick,->](0,0)--(1,-1);
\end{tikzpicture}
&
\begin{tikzpicture}[scale=.2, baseline=(current bounding box.center)]

\draw[thick,->](0,0)--(-1,1);
\draw[thick,->](0,0)--(0,1);
\draw[thick,->](0,0)--(1,0);
\draw[thick,->](0,0)--(0,-1);
\draw[thick,->](0,0)--(1,-1);
\end{tikzpicture}
&
\begin{tikzpicture}[scale=.2, baseline=(current bounding box.center)]

\draw[thick,->](0,0)--(0,1);
\draw[thick,->](0,0)--(1,1);
\draw[thick,->](0,0)--(-1,0);
\draw[thick,->](0,0)--(0,-1);
\end{tikzpicture}
&
\begin{tikzpicture}[scale=.2, baseline=(current bounding box.center)]

\draw[thick,->](0,0)--(0,1);
\draw[thick,->](0,0)--(1,1);
\draw[thick,->](0,0)--(-1,0);
\draw[thick,->](0,0)--(-1,-1);
\draw[thick,->](0,0)--(0,-1);
\end{tikzpicture}
&
\begin{tikzpicture}[scale=.2, baseline=(current bounding box.center)]

\draw[thick,->](0,0)--(0,1);
\draw[thick,->](0,0)--(1,1);
\draw[thick,->](0,0)--(-1,0);
\draw[thick,->](0,0)--(1,0);
\draw[thick,->](0,0)--(-1,-1);
\end{tikzpicture} 
&
\begin{tikzpicture}[scale=.2, baseline=(current bounding box.center)]

\draw[thick,->](0,0)--(0,1);
\draw[thick,->](0,0)--(1,1);
\draw[thick,->](0,0)--(-1,0);
\draw[thick,->](0,0)--(1,0);
\draw[thick,->](0,0)--(0,-1);
\end{tikzpicture}\end{array}$

\item
\begin{center} D-transcendental cases \end{center}
$\begin{array}{llllllllllllllll}
\begin{tikzpicture}[scale=.2, baseline=(current bounding box.center)]

\draw[thick,->](0,0)--(-1,1);
\draw[thick,->](0,0)--(1,1);
\draw[thick,->](0,0)--(0,-1);
\draw[thick,->](0,0)--(1,-1);
\end{tikzpicture}
& 
\begin{tikzpicture}[scale=.2, baseline=(current bounding box.center)]

\draw[thick,->](0,0)--(-1,1);
\draw[thick,->](0,0)--(0,1);
\draw[thick,->](0,0)--(1,1);
\draw[thick,->](0,0)--(0,-1);
\draw[thick,->](0,0)--(1,-1);
\end{tikzpicture}
&
\begin{tikzpicture}[scale=.2, baseline=(current bounding box.center)]

\draw[thick,->](0,0)--(-1,1);
\draw[thick,->](0,0)--(0,1);
\draw[thick,->](0,0)--(1,1);
\draw[thick,->](0,0)--(-1,0);
\draw[thick,->](0,0)--(1,-1);
\end{tikzpicture}
&
\begin{tikzpicture}[scale=.2, baseline=(current bounding box.center)]

\draw[thick,->](0,0)--(-1,1);
\draw[thick,->](0,0)--(1,1);
\draw[thick,->](0,0)--(-1,0);
\draw[thick,->](0,0)--(0,-1);
\draw[thick,->](0,0)--(1,-1);
\end{tikzpicture}
&
\begin{tikzpicture}[scale=.2, baseline=(current bounding box.center)]

\draw[thick,->](0,0)--(-1,1);
\draw[thick,->](0,0)--(0,1);
\draw[thick,->](0,0)--(1,1);
\draw[thick,->](0,0)--(-1,0);
\draw[thick,->](0,0)--(1,0);
\draw[thick,->](0,0)--(1,-1);
\end{tikzpicture}
&
\begin{tikzpicture}[scale=.2, baseline=(current bounding box.center)]

\draw[thick,->](0,0)--(-1,1);
\draw[thick,->](0,0)--(0,1);
\draw[thick,->](0,0)--(1,1);
\draw[thick,->](0,0)--(1,0);
\draw[thick,->](0,0)--(-1,-1);
\draw[thick,->](0,0)--(1,-1);
\end{tikzpicture}
&
\begin{tikzpicture}[scale=.2, baseline=(current bounding box.center)]

\draw[thick,->](0,0)--(-1,1);
\draw[thick,->](0,0)--(0,1);
\draw[thick,->](0,0)--(1,1);
\draw[thick,->](0,0)--(-1,0);
\draw[thick,->](0,0)--(0,-1);
\draw[thick,->](0,0)--(1,-1);
\end{tikzpicture}
&
\begin{tikzpicture}[scale=.2, baseline=(current bounding box.center)]

\draw[thick,->](0,0)--(-1,1);
\draw[thick,->](0,0)--(0,1);
\draw[thick,->](0,0)--(1,1);
\draw[thick,->](0,0)--(-1,0);
\draw[thick,->](0,0)--(-1,-1);
\draw[thick,->](0,0)--(1,-1);
\end{tikzpicture}
&
\begin{tikzpicture}[scale=.2, baseline=(current bounding box.center)]

\draw[thick,->](0,0)--(-1,1);
\draw[thick,->](0,0)--(1,1);
\draw[thick,->](0,0)--(-1,0);
\draw[thick,->](0,0)--(-1,-1);
\draw[thick,->](0,0)--(0,-1);
\draw[thick,->](0,0)--(1,-1);
\end{tikzpicture}
&
\begin{tikzpicture}[scale=.2, baseline=(current bounding box.center)]

\draw[thick,->](0,0)--(-1,1);
\draw[thick,->](0,0)--(0,1);
\draw[thick,->](0,0)--(1,1);
\draw[thick,->](0,0)--(-1,0);
\draw[thick,->](0,0)--(1,0);
\draw[thick,->](0,0)--(0,-1);
\draw[thick,->](0,0)--(1,-1);
\end{tikzpicture}
&
\begin{tikzpicture}[scale=.2, baseline=(current bounding box.center)]

\draw[thick,->](0,0)--(-1,1);
\draw[thick,->](0,0)--(1,1);
\draw[thick,->](0,0)--(-1,-1);
\draw[thick,->](0,0)--(0,-1);
\end{tikzpicture}
&
\begin{tikzpicture}[scale=.2, baseline=(current bounding box.center)]

\draw[thick,->](0,0)--(-1,1);
\draw[thick,->](0,0)--(1,1);
\draw[thick,->](0,0)--(-1,0);
\draw[thick,->](0,0)--(0,-1);
\end{tikzpicture}
&
\begin{tikzpicture}[scale=.2, baseline=(current bounding box.center)]

\draw[thick,->](0,0)--(-1,1);
\draw[thick,->](0,0)--(0,1);
\draw[thick,->](0,0)--(1,1);
\draw[thick,->](0,0)--(-1,-1);
\draw[thick,->](0,0)--(0,-1);
\end{tikzpicture}
&
\begin{tikzpicture}[scale=.2, baseline=(current bounding box.center)]

\draw[thick,->](0,0)--(-1,1);
\draw[thick,->](0,0)--(0,1);
\draw[thick,->](0,0)--(1,1);
\draw[thick,->](0,0)--(-1,0);
\draw[thick,->](0,0)--(0,-1);
\end{tikzpicture}
&
\begin{tikzpicture}[scale=.2, baseline=(current bounding box.center)]

\draw[thick,->](0,0)--(-1,1);
\draw[thick,->](0,0)--(1,1);
\draw[thick,->](0,0)--(-1,0);
\draw[thick,->](0,0)--(-1,-1);
\draw[thick,->](0,0)--(0,-1);
\end{tikzpicture}
&
\begin{tikzpicture}[scale=.2, baseline=(current bounding box.center)]

\draw[thick,->](0,0)--(-1,1);
\draw[thick,->](0,0)--(0,1);
\draw[thick,->](0,0)--(1,1);
\draw[thick,->](0,0)--(-1,0);
\draw[thick,->](0,0)--(-1,-1);
\draw[thick,->](0,0)--(0,-1);
\end{tikzpicture}
\\
\begin{tikzpicture}[scale=.2, baseline=(current bounding box.center)]

\draw[thick,->](0,0)--(-1,1);
\draw[thick,->](0,0)--(0,1);
\draw[thick,->](0,0)--(1,1);
\draw[thick,->](0,0)--(1,0);
\draw[thick,->](0,0)--(-1,-1);
\end{tikzpicture}
&
\begin{tikzpicture}[scale=.2, baseline=(current bounding box.center)]

\draw[thick,->](0,0)--(-1,1);
\draw[thick,->](0,0)--(0,1);
\draw[thick,->](0,0)--(1,1);
\draw[thick,->](0,0)--(1,0);
\draw[thick,->](0,0)--(0,-1);
\end{tikzpicture}
&
\begin{tikzpicture}[scale=.2, baseline=(current bounding box.center)]

\draw[thick,->](0,0)--(-1,1);
\draw[thick,->](0,0)--(0,1);
\draw[thick,->](0,0)--(1,1);
\draw[thick,->](0,0)--(-1,0);
\draw[thick,->](0,0)--(1,0);
\draw[thick,->](0,0)--(-1,-1);
\draw[thick,->](0,0)--(0,-1);
\end{tikzpicture}
&
\begin{tikzpicture}[scale=.2, baseline=(current bounding box.center)]

\draw[thick,->](0,0)--(0,1);
\draw[thick,->](0,0)--(1,1);
\draw[thick,->](0,0)--(-1,-1);
\draw[thick,->](0,0)--(1,-1);
\end{tikzpicture}
&
\begin{tikzpicture}[scale=.2, baseline=(current bounding box.center)]

\draw[thick,->](0,0)--(0,1);
\draw[thick,->](0,0)--(1,1);
\draw[thick,->](0,0)--(-1,0);
\draw[thick,->](0,0)--(1,-1);
\end{tikzpicture}
&
\begin{tikzpicture}[scale=.2, baseline=(current bounding box.center)]

\draw[thick,->](0,0)--(0,1);
\draw[thick,->](0,0)--(1,1);
\draw[thick,->](0,0)--(-1,0);
\draw[thick,->](0,0)--(-1,-1);
\draw[thick,->](0,0)--(1,-1);
\end{tikzpicture}
&
\begin{tikzpicture}[scale=.2, baseline=(current bounding box.center)]

\draw[thick,->](0,0)--(0,1);
\draw[thick,->](0,0)--(1,1);
\draw[thick,->](0,0)--(-1,-1);
\draw[thick,->](0,0)--(0,-1);
\draw[thick,->](0,0)--(1,-1);
\end{tikzpicture}
&
\begin{tikzpicture}[scale=.2, baseline=(current bounding box.center)]

\draw[thick,->](0,0)--(0,1);
\draw[thick,->](0,0)--(1,1);
\draw[thick,->](0,0)--(1,0);
\draw[thick,->](0,0)--(-1,-1);
\draw[thick,->](0,0)--(0,-1);
\draw[thick,->](0,0)--(1,-1);
\end{tikzpicture}
&
\begin{tikzpicture}[scale=.2, baseline=(current bounding box.center)]

\draw[thick,->](0,0)--(0,1);
\draw[thick,->](0,0)--(1,1);
\draw[thick,->](0,0)--(-1,0);
\draw[thick,->](0,0)--(1,0);
\draw[thick,->](0,0)--(-1,-1);
\draw[thick,->](0,0)--(1,-1);
\end{tikzpicture}
&
 \begin{tikzpicture}[scale=.2, baseline=(current bounding box.center)]

\draw[thick,->](0,0)--(0,1);
\draw[thick,->](0,0)--(1,1);
\draw[thick,->](0,0)--(-1,0);
\draw[thick,->](0,0)--(-1,-1);
\draw[thick,->](0,0)--(0,-1);
\draw[thick,->](0,0)--(1,-1);
\end{tikzpicture}
&
\begin{tikzpicture}[scale=.2, baseline=(current bounding box.center)]

\draw[thick,->](0,0)--(0,1);
\draw[thick,->](0,0)--(-1,0);
\draw[thick,->](0,0)--(1,0);
\draw[thick,->](0,0)--(0,-1);
\draw[thick,->](0,0)--(1,-1);
\end{tikzpicture}
&
\begin{tikzpicture}[scale=.2, baseline=(current bounding box.center)]

\draw[thick,->](0,0)--(0,1);
\draw[thick,->](0,0)--(-1,0);
\draw[thick,->](0,0)--(1,0);
\draw[thick,->](0,0)--(-1,-1);
\draw[thick,->](0,0)--(0,-1);
\end{tikzpicture}
&
\begin{tikzpicture}[scale=.2, baseline=(current bounding box.center)]

\draw[thick,->](0,0)--(-1,1);
\draw[thick,->](0,0)--(0,1);
\draw[thick,->](0,0)--(1,0);
\draw[thick,->](0,0)--(-1,-1);
\draw[thick,->](0,0)--(1,-1);
\end{tikzpicture}
&
\begin{tikzpicture}[scale=.2, baseline=(current bounding box.center)]

\draw[thick,->](0,0)--(-1,1);
\draw[thick,->](0,0)--(0,1);
\draw[thick,->](0,0)--(-1,0);
\draw[thick,->](0,0)--(1,0);
\draw[thick,->](0,0)--(-1,-1);
\draw[thick,->](0,0)--(1,-1);
\end{tikzpicture}
&
\begin{tikzpicture}[scale=.2, baseline=(current bounding box.center)]

\draw[thick,->](0,0)--(-1,1);
\draw[thick,->](0,0)--(0,1);
\draw[thick,->](0,0)--(-1,0);
\draw[thick,->](0,0)--(1,0);
\draw[thick,->](0,0)--(-1,-1);
\draw[thick,->](0,0)--(0,-1);
\draw[thick,->](0,0)--(1,-1);
\end{tikzpicture}
&
\begin{tikzpicture}[scale=.2, baseline=(current bounding box.center)]

\draw[thick,->](0,0)--(0,1);
\draw[thick,->](0,0)--(-1,0);
\draw[thick,->](0,0)--(0,-1);
\draw[thick,->](0,0)--(1,-1);
\end{tikzpicture}
\\
\begin{tikzpicture}[scale=.2, baseline=(current bounding box.center)]

\draw[thick,->](0,0)--(0,1);
\draw[thick,->](0,0)--(-1,0);
\draw[thick,->](0,0)--(-1,-1);
\draw[thick,->](0,0)--(1,-1);
\end{tikzpicture}
&
\begin{tikzpicture}[scale=.2, baseline=(current bounding box.center)]

\draw[thick,->](0,0)--(-1,1);
\draw[thick,->](0,0)--(0,1);
\draw[thick,->](0,0)--(-1,-1);
\draw[thick,->](0,0)--(1,-1);
\end{tikzpicture}
&
\begin{tikzpicture}[scale=.2, baseline=(current bounding box.center)]

\draw[thick,->](0,0)--(-1,1);
\draw[thick,->](0,0)--(0,1);
\draw[thick,->](0,0)--(-1,0);
\draw[thick,->](0,0)--(1,-1);
\end{tikzpicture}
&
\begin{tikzpicture}[scale=.2, baseline=(current bounding box.center)]

\draw[thick,->](0,0)--(0,1);
\draw[thick,->](0,0)--(-1,0);
\draw[thick,->](0,0)--(-1,-1);
\draw[thick,->](0,0)--(0,-1);
\draw[thick,->](0,0)--(1,-1);
\end{tikzpicture}
&
\begin{tikzpicture}[scale=.2, baseline=(current bounding box.center)]

\draw[thick,->](0,0)--(-1,1);
\draw[thick,->](0,0)--(0,1);
\draw[thick,->](0,0)--(-1,-1);
\draw[thick,->](0,0)--(0,-1);
\draw[thick,->](0,0)--(1,-1);
\end{tikzpicture}
&
 \begin{tikzpicture}[scale=.2, baseline=(current bounding box.center)]

\draw[thick,->](0,0)--(-1,1);
\draw[thick,->](0,0)--(0,1);
\draw[thick,->](0,0)--(-1,0);
\draw[thick,->](0,0)--(0,-1);
\draw[thick,->](0,0)--(1,-1);
\end{tikzpicture}
&
 \begin{tikzpicture}[scale=.2, baseline=(current bounding box.center)]

\draw[thick,->](0,0)--(-1,1);
\draw[thick,->](0,0)--(0,1);
\draw[thick,->](0,0)--(-1,0);
\draw[thick,->](0,0)--(-1,-1);
\draw[thick,->](0,0)--(1,-1);
\end{tikzpicture}
&
\begin{tikzpicture}[scale=.2, baseline=(current bounding box.center)]

\draw[thick,->](0,0)--(-1,1);
\draw[thick,->](0,0)--(0,1);
\draw[thick,->](0,0)--(-1,0);
\draw[thick,->](0,0)--(-1,-1);
\draw[thick,->](0,0)--(0,-1);
\draw[thick,->](0,0)--(1,-1);
\end{tikzpicture}
& 
\begin{tikzpicture}[scale=.2, baseline=(current bounding box.center)]

\draw[thick,->](0,0)--(1,1);
\draw[thick,->](0,0)--(-1,0);
\draw[thick,->](0,0)--(-1,-1);
\draw[thick,->](0,0)--(0,-1);
\end{tikzpicture}
&
\begin{tikzpicture}[scale=.2, baseline=(current bounding box.center)]

\draw[thick,->](0,0)--(0,1);
\draw[thick,->](0,0)--(1,1);
\draw[thick,->](0,0)--(1,0);
\draw[thick,->](0,0)--(-1,-1);
\end{tikzpicture}
& 
\begin{tikzpicture}[scale=.2, baseline=(current bounding box.center)]

\draw[thick,->](0,0)--(-1,1);
\draw[thick,->](0,0)--(0,1);
\draw[thick,->](0,0)--(1,1);
\draw[thick,->](0,0)--(1,0);
\draw[thick,->](0,0)--(1,-1);
\end{tikzpicture} 
&
\begin{tikzpicture}[scale=.2, baseline=(current bounding box.center)]

\draw[thick,->](0,0)--(-1,1);
\draw[thick,->](0,0)--(0,1);
\draw[thick,->](0,0)--(1,0);
\draw[thick,->](0,0)--(1,-1);
\end{tikzpicture} 
&
\begin{tikzpicture}[scale=.2, baseline=(current bounding box.center)]

\draw[thick,->](0,0)--(-1,1);
\draw[thick,->](0,0)--(0,1);
\draw[thick,->](0,0)--(1,-1);
\end{tikzpicture} 
& 
\begin{tikzpicture}[scale=.2, baseline=(current bounding box.center)]

\draw[thick,->](0,0)--(-1,1);
\draw[thick,->](0,0)--(1,1);
\draw[thick,->](0,0)--(1,-1);
\end{tikzpicture} 
&
\begin{tikzpicture}[scale=.2, baseline=(current bounding box.center)]

\draw[thick,->](0,0)--(-1,1);
\draw[thick,->](0,0)--(0,1);
\draw[thick,->](0,0)--(1,1);
\draw[thick,->](0,0)--(1,-1);
\end{tikzpicture}&
\end{array}$
\end{trivlist}
\caption{Classification of the $79$ models of walks in the quarter plane. The algebraic and D-finite cases correspond to walks with a finite group.}\label{figcas}
\end{figure}
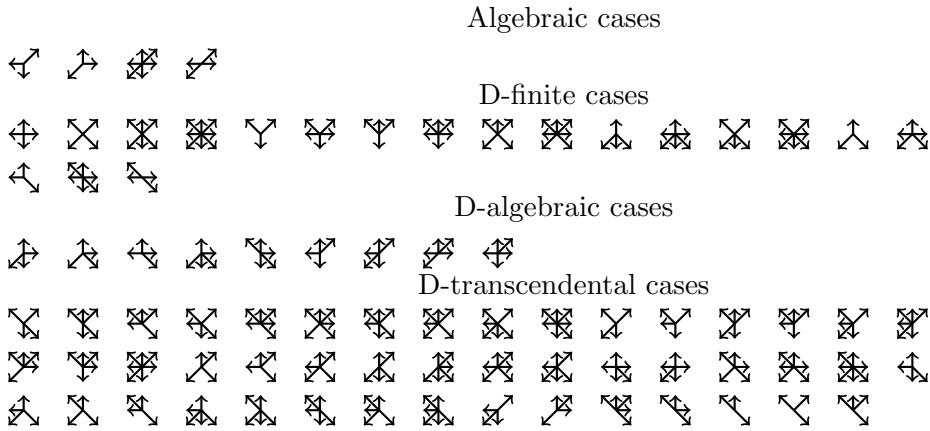

The next step is the study of walks in the three quarter plane. This question is not vacuous, since every walk confined to a cone whose apex is the origin, and whose opening angle is $\alpha 2\pi$, with $\alpha\in (0,1)\cap\Q$, is equivalent to a walk, with possibly large steps, confined to the half plane (when $\alpha=1/2$), quarter plane (when the cone is convex), or three quarter plane (when the cone is concave). Similarly to the quarter plane case, the study of walks with large steps is far from being reachable in full generality. Let us focus on the models with small steps.  The $5$ genus zero configurations, which correspond to the situation where the step set is included in $\begin{tikzpicture}[scale=.2, baseline=(current bounding box.center)]

\draw[thick,->](0,0)--(-1,1);
\draw[thick,->](0,0)--(0,1);
\draw[thick,->](0,0)--(1,1);
\draw[thick,->](0,0)--(1,0);
\draw[thick,->](0,0)--(1,-1);
\end{tikzpicture}$,
are not constrained by the three quarter plane restriction, they are analogous to a walk model in the plane and therefore have a rational generating function.
In \cite{BM-16}, Bousquet-M\'elou proves that the simple walks (with step set $\begin{tikzpicture}[scale=.2, baseline=(current bounding box.center)]
\draw[thick,->](0,0)--(0,1);
\draw[thick,->](0,0)--(0,-1);
\draw[thick,->](0,0)--(1,0);
\draw[thick,->](0,0)--(-1,0);
\end{tikzpicture}$) and the diagonal walks (with step set $\begin{tikzpicture}[scale=.2, baseline=(current bounding box.center)]
\draw[thick,->](0,0)--(1,1);
\draw[thick,->](0,0)--(-1,-1);
\draw[thick,->](0,0)--(1,-1);
\draw[thick,->](0,0)--(-1,1);
\end{tikzpicture}$) have  D-finite generating series. A few years later, with the same methods, Bousquet-M\'elou and Wallner \cite{BMW-20} prove that the king walks (with step set $\begin{tikzpicture}[scale=.2, baseline=(current bounding box.center)]
\draw[thick,->](0,0)--(0,1);
\draw[thick,->](0,0)--(0,-1);
\draw[thick,->](0,0)--(1,0);
\draw[thick,->](0,0)--(-1,0);
\draw[thick,->](0,0)--(1,1);
\draw[thick,->](0,0)--(-1,-1);
\draw[thick,->](0,0)--(1,-1);
\draw[thick,->](0,0)--(-1,1);
\end{tikzpicture}$)
also have  D-finite generating series. Using an original connection with planar maps, Budd~\cite[Sec.~3.3]{Bu-20} obtains various enumerating formulas for planar walks, keeping track of the widing  angle. In \cite{mustapha2019non}, Mustapha proves that the generating series for genus one walks with infinite group is not D-finite. Finally, in \cite{RaTr-18}, using integral expression of the generating series, the authors prove that four diagonally symmetric models with finite group have a D-finite generating series. Figure \ref{figcas34} summarizes the current classification of walks in the three quarter plane. 

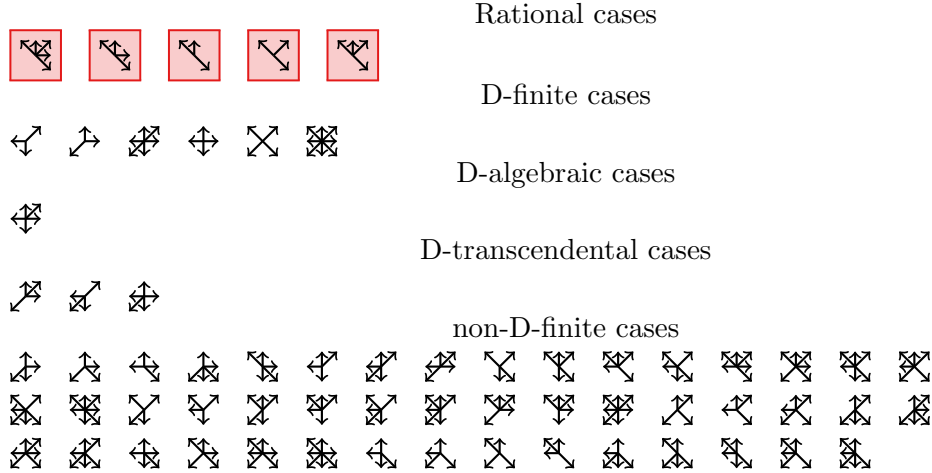
\begin{figure}[!h]
\begin{trivlist}
\item  \begin{center}
Rational cases
\end{center}
$\begin{array}{lllll}
\setlength{\fboxrule}{0.75pt}
\fcolorbox{nred!90}{nred!20}{$
\begin{tikzpicture}[scale=.2, baseline=(current bounding box.center)]

\draw[thick,->](0,0)--(-1,1);
\draw[thick,->](0,0)--(0,1);
\draw[thick,->](0,0)--(1,1);
\draw[thick,->](0,0)--(1,0);
\draw[thick,->](0,0)--(1,-1);
\end{tikzpicture} $}
&
\setlength{\fboxrule}{0.75pt}
\fcolorbox{nred!90}{nred!20}{$
\begin{tikzpicture}[scale=.2, baseline=(current bounding box.center)]

\draw[thick,->](0,0)--(-1,1);
\draw[thick,->](0,0)--(0,1);
\draw[thick,->](0,0)--(1,0);
\draw[thick,->](0,0)--(1,-1);
\end{tikzpicture} $}
&
\setlength{\fboxrule}{0.75pt}
\fcolorbox{nred!90}{nred!20}{$
\begin{tikzpicture}[scale=.2, baseline=(current bounding box.center)]

\draw[thick,->](0,0)--(-1,1);
\draw[thick,->](0,0)--(0,1);
\draw[thick,->](0,0)--(1,-1);
\end{tikzpicture} $}
& 
\setlength{\fboxrule}{0.75pt}
\fcolorbox{nred!90}{nred!20}{$
\begin{tikzpicture}[scale=.2, baseline=(current bounding box.center)]

\draw[thick,->](0,0)--(-1,1);
\draw[thick,->](0,0)--(1,1);
\draw[thick,->](0,0)--(1,-1);
\end{tikzpicture} $}
&
\setlength{\fboxrule}{0.75pt}
\fcolorbox{nred!90}{nred!20}{$
\begin{tikzpicture}[scale=.2, baseline=(current bounding box.center)]

\draw[thick,->](0,0)--(-1,1);
\draw[thick,->](0,0)--(0,1);
\draw[thick,->](0,0)--(1,1);
\draw[thick,->](0,0)--(1,-1);
\end{tikzpicture}$}
\end{array}$

\item \begin{center} D-finite cases \end{center}

$\begin{array}{llllll}
\begin{tikzpicture}[scale=.2, baseline=(current bounding box.center)]

\draw[thick,->](0,0)--(1,1);
\draw[thick,->](0,0)--(-1,0);
\draw[thick,->](0,0)--(0,-1);
\end{tikzpicture}
&
\begin{tikzpicture}[scale=.2, baseline=(current bounding box.center)]

\draw[thick,->](0,0)--(0,1);
\draw[thick,->](0,0)--(1,0);
\draw[thick,->](0,0)--(-1,-1);
\end{tikzpicture}
&
\begin{tikzpicture}[scale=.2, baseline=(current bounding box.center)]

\draw[thick,->](0,0)--(0,1);
\draw[thick,->](0,0)--(1,1);
\draw[thick,->](0,0)--(-1,0);
\draw[thick,->](0,0)--(1,0);
\draw[thick,->](0,0)--(-1,-1);
\draw[thick,->](0,0)--(0,-1);
\end{tikzpicture}
&
\begin{tikzpicture}[scale=.2, baseline=(current bounding box.center)]

\draw[thick,->](0,0)--(0,1);
\draw[thick,->](0,0)--(-1,0);
\draw[thick,->](0,0)--(1,0);
\draw[thick,->](0,0)--(0,-1);
\end{tikzpicture}
&
\begin{tikzpicture}[scale=.2, baseline=(current bounding box.center)]

\draw[thick,->](0,0)--(-1,1);
\draw[thick,->](0,0)--(1,1);
\draw[thick,->](0,0)--(-1,-1);
\draw[thick,->](0,0)--(1,-1);
\end{tikzpicture}
&
\begin{tikzpicture}[scale=.2, baseline=(current bounding box.center)]

\draw[thick,->](0,0)--(-1,1);
\draw[thick,->](0,0)--(0,1);
\draw[thick,->](0,0)--(1,1);
\draw[thick,->](0,0)--(-1,0);
\draw[thick,->](0,0)--(1,0);
\draw[thick,->](0,0)--(-1,-1);
\draw[thick,->](0,0)--(0,-1);
\draw[thick,->](0,0)--(1,-1);
\end{tikzpicture}

\end{array}$

\item \begin{center} D-algebraic cases \end{center}

$\begin{array}{l}
\begin{tikzpicture}[scale=.2, baseline=(current bounding box.center)]

\draw[thick,->](0,0)--(0,1);
\draw[thick,->](0,0)--(1,1);
\draw[thick,->](0,0)--(-1,0);
\draw[thick,->](0,0)--(1,0);
\draw[thick,->](0,0)--(0,-1);
\end{tikzpicture}

\end{array}$

\item \begin{center} D-transcendental cases \end{center}

$\begin{array}{lll}
\begin{tikzpicture}[scale=.2, baseline=(current bounding box.center)]
\draw[thick,->](0,0)--(0,1);
\draw[thick,->](0,0)--(1,1);
\draw[thick,->](0,0)--(1,0);
\draw[thick,->](0,0)--(-1,-1);
\end{tikzpicture}
&
\setlength{\fboxrule}{0.75pt}
\begin{tikzpicture}[scale=.2, baseline=(current bounding box.center)]
\draw[thick,->](0,0)--(1,1);
\draw[thick,->](0,0)--(-1,0);
\draw[thick,->](0,0)--(-1,-1);
\draw[thick,->](0,0)--(0,-1);
\end{tikzpicture}
&
\begin{tikzpicture}[scale=.2, baseline=(current bounding box.center)]

\draw[thick,->](0,0)--(0,1);
\draw[thick,->](0,0)--(-1,0);
\draw[thick,->](0,0)--(1,0);
\draw[thick,->](0,0)--(-1,-1);
\draw[thick,->](0,0)--(0,-1);
\end{tikzpicture}

\end{array}$

\item \begin{center} non-D-finite cases \end{center}

$\begin{array}{llllllllllllllll}

\begin{tikzpicture}[scale=.2, baseline=(current bounding box.center)]

\draw[thick,->](0,0)--(0,1);
\draw[thick,->](0,0)--(1,0);
\draw[thick,->](0,0)--(-1,-1);
\draw[thick,->](0,0)--(0,-1);
\end{tikzpicture}
&
\begin{tikzpicture}[scale=.2, baseline=(current bounding box.center)]

\draw[thick,->](0,0)--(0,1);
\draw[thick,->](0,0)--(1,0);
\draw[thick,->](0,0)--(-1,-1);
\draw[thick,->](0,0)--(1,-1);
\end{tikzpicture}
&
\begin{tikzpicture}[scale=.2, baseline=(current bounding box.center)]

\draw[thick,->](0,0)--(0,1);
\draw[thick,->](0,0)--(-1,0);
\draw[thick,->](0,0)--(1,0);
\draw[thick,->](0,0)--(1,-1);
\end{tikzpicture}
&
\begin{tikzpicture}[scale=.2, baseline=(current bounding box.center)]

\draw[thick,->](0,0)--(0,1);
\draw[thick,->](0,0)--(1,0);
\draw[thick,->](0,0)--(-1,-1);
\draw[thick,->](0,0)--(0,-1);
\draw[thick,->](0,0)--(1,-1);
\end{tikzpicture}
&
\begin{tikzpicture}[scale=.2, baseline=(current bounding box.center)]

\draw[thick,->](0,0)--(-1,1);
\draw[thick,->](0,0)--(0,1);
\draw[thick,->](0,0)--(1,0);
\draw[thick,->](0,0)--(0,-1);
\draw[thick,->](0,0)--(1,-1);
\end{tikzpicture}
&
\begin{tikzpicture}[scale=.2, baseline=(current bounding box.center)]

\draw[thick,->](0,0)--(0,1);
\draw[thick,->](0,0)--(1,1);
\draw[thick,->](0,0)--(-1,0);
\draw[thick,->](0,0)--(0,-1);
\end{tikzpicture}
&
\begin{tikzpicture}[scale=.2, baseline=(current bounding box.center)]

\draw[thick,->](0,0)--(0,1);
\draw[thick,->](0,0)--(1,1);
\draw[thick,->](0,0)--(-1,0);
\draw[thick,->](0,0)--(-1,-1);
\draw[thick,->](0,0)--(0,-1);
\end{tikzpicture}
&
\begin{tikzpicture}[scale=.2, baseline=(current bounding box.center)]

\draw[thick,->](0,0)--(0,1);
\draw[thick,->](0,0)--(1,1);
\draw[thick,->](0,0)--(-1,0);
\draw[thick,->](0,0)--(1,0);
\draw[thick,->](0,0)--(-1,-1);
\end{tikzpicture} 
&
\begin{tikzpicture}[scale=.2, baseline=(current bounding box.center)]

\draw[thick,->](0,0)--(-1,1);
\draw[thick,->](0,0)--(1,1);
\draw[thick,->](0,0)--(0,-1);
\draw[thick,->](0,0)--(1,-1);
\end{tikzpicture}
& 
\begin{tikzpicture}[scale=.2, baseline=(current bounding box.center)]

\draw[thick,->](0,0)--(-1,1);
\draw[thick,->](0,0)--(0,1);
\draw[thick,->](0,0)--(1,1);
\draw[thick,->](0,0)--(0,-1);
\draw[thick,->](0,0)--(1,-1);
\end{tikzpicture}
&
\begin{tikzpicture}[scale=.2, baseline=(current bounding box.center)]

\draw[thick,->](0,0)--(-1,1);
\draw[thick,->](0,0)--(0,1);
\draw[thick,->](0,0)--(1,1);
\draw[thick,->](0,0)--(-1,0);
\draw[thick,->](0,0)--(1,-1);
\end{tikzpicture}
&
\begin{tikzpicture}[scale=.2, baseline=(current bounding box.center)]

\draw[thick,->](0,0)--(-1,1);
\draw[thick,->](0,0)--(1,1);
\draw[thick,->](0,0)--(-1,0);
\draw[thick,->](0,0)--(0,-1);
\draw[thick,->](0,0)--(1,-1);
\end{tikzpicture}
&
\begin{tikzpicture}[scale=.2, baseline=(current bounding box.center)]

\draw[thick,->](0,0)--(-1,1);
\draw[thick,->](0,0)--(0,1);
\draw[thick,->](0,0)--(1,1);
\draw[thick,->](0,0)--(-1,0);
\draw[thick,->](0,0)--(1,0);
\draw[thick,->](0,0)--(1,-1);
\end{tikzpicture}
&
\begin{tikzpicture}[scale=.2, baseline=(current bounding box.center)]

\draw[thick,->](0,0)--(-1,1);
\draw[thick,->](0,0)--(0,1);
\draw[thick,->](0,0)--(1,1);
\draw[thick,->](0,0)--(1,0);
\draw[thick,->](0,0)--(-1,-1);
\draw[thick,->](0,0)--(1,-1);
\end{tikzpicture}
&
\begin{tikzpicture}[scale=.2, baseline=(current bounding box.center)]

\draw[thick,->](0,0)--(-1,1);
\draw[thick,->](0,0)--(0,1);
\draw[thick,->](0,0)--(1,1);
\draw[thick,->](0,0)--(-1,0);
\draw[thick,->](0,0)--(0,-1);
\draw[thick,->](0,0)--(1,-1);
\end{tikzpicture}
&
\begin{tikzpicture}[scale=.2, baseline=(current bounding box.center)]

\draw[thick,->](0,0)--(-1,1);
\draw[thick,->](0,0)--(0,1);
\draw[thick,->](0,0)--(1,1);
\draw[thick,->](0,0)--(-1,0);
\draw[thick,->](0,0)--(-1,-1);
\draw[thick,->](0,0)--(1,-1);
\end{tikzpicture}
\\
\begin{tikzpicture}[scale=.2, baseline=(current bounding box.center)]

\draw[thick,->](0,0)--(-1,1);
\draw[thick,->](0,0)--(1,1);
\draw[thick,->](0,0)--(-1,0);
\draw[thick,->](0,0)--(-1,-1);
\draw[thick,->](0,0)--(0,-1);
\draw[thick,->](0,0)--(1,-1);
\end{tikzpicture}
&
\begin{tikzpicture}[scale=.2, baseline=(current bounding box.center)]

\draw[thick,->](0,0)--(-1,1);
\draw[thick,->](0,0)--(0,1);
\draw[thick,->](0,0)--(1,1);
\draw[thick,->](0,0)--(-1,0);
\draw[thick,->](0,0)--(1,0);
\draw[thick,->](0,0)--(0,-1);
\draw[thick,->](0,0)--(1,-1);
\end{tikzpicture}
&
\begin{tikzpicture}[scale=.2, baseline=(current bounding box.center)]

\draw[thick,->](0,0)--(-1,1);
\draw[thick,->](0,0)--(1,1);
\draw[thick,->](0,0)--(-1,-1);
\draw[thick,->](0,0)--(0,-1);
\end{tikzpicture}
&
\begin{tikzpicture}[scale=.2, baseline=(current bounding box.center)]

\draw[thick,->](0,0)--(-1,1);
\draw[thick,->](0,0)--(1,1);
\draw[thick,->](0,0)--(-1,0);
\draw[thick,->](0,0)--(0,-1);
\end{tikzpicture}
&
\begin{tikzpicture}[scale=.2, baseline=(current bounding box.center)]

\draw[thick,->](0,0)--(-1,1);
\draw[thick,->](0,0)--(0,1);
\draw[thick,->](0,0)--(1,1);
\draw[thick,->](0,0)--(-1,-1);
\draw[thick,->](0,0)--(0,-1);
\end{tikzpicture}
&
\begin{tikzpicture}[scale=.2, baseline=(current bounding box.center)]

\draw[thick,->](0,0)--(-1,1);
\draw[thick,->](0,0)--(0,1);
\draw[thick,->](0,0)--(1,1);
\draw[thick,->](0,0)--(-1,0);
\draw[thick,->](0,0)--(0,-1);
\end{tikzpicture}
&
\begin{tikzpicture}[scale=.2, baseline=(current bounding box.center)]

\draw[thick,->](0,0)--(-1,1);
\draw[thick,->](0,0)--(1,1);
\draw[thick,->](0,0)--(-1,0);
\draw[thick,->](0,0)--(-1,-1);
\draw[thick,->](0,0)--(0,-1);
\end{tikzpicture}
&
\begin{tikzpicture}[scale=.2, baseline=(current bounding box.center)]

\draw[thick,->](0,0)--(-1,1);
\draw[thick,->](0,0)--(0,1);
\draw[thick,->](0,0)--(1,1);
\draw[thick,->](0,0)--(-1,0);
\draw[thick,->](0,0)--(-1,-1);
\draw[thick,->](0,0)--(0,-1);
\end{tikzpicture}
&
\begin{tikzpicture}[scale=.2, baseline=(current bounding box.center)]

\draw[thick,->](0,0)--(-1,1);
\draw[thick,->](0,0)--(0,1);
\draw[thick,->](0,0)--(1,1);
\draw[thick,->](0,0)--(1,0);
\draw[thick,->](0,0)--(-1,-1);
\end{tikzpicture}
&
\begin{tikzpicture}[scale=.2, baseline=(current bounding box.center)]

\draw[thick,->](0,0)--(-1,1);
\draw[thick,->](0,0)--(0,1);
\draw[thick,->](0,0)--(1,1);
\draw[thick,->](0,0)--(1,0);
\draw[thick,->](0,0)--(0,-1);
\end{tikzpicture}
&
\begin{tikzpicture}[scale=.2, baseline=(current bounding box.center)]

\draw[thick,->](0,0)--(-1,1);
\draw[thick,->](0,0)--(0,1);
\draw[thick,->](0,0)--(1,1);
\draw[thick,->](0,0)--(-1,0);
\draw[thick,->](0,0)--(1,0);
\draw[thick,->](0,0)--(-1,-1);
\draw[thick,->](0,0)--(0,-1);
\end{tikzpicture}
&
\begin{tikzpicture}[scale=.2, baseline=(current bounding box.center)]

\draw[thick,->](0,0)--(0,1);
\draw[thick,->](0,0)--(1,1);
\draw[thick,->](0,0)--(-1,-1);
\draw[thick,->](0,0)--(1,-1);
\end{tikzpicture}
&
\begin{tikzpicture}[scale=.2, baseline=(current bounding box.center)]

\draw[thick,->](0,0)--(0,1);
\draw[thick,->](0,0)--(1,1);
\draw[thick,->](0,0)--(-1,0);
\draw[thick,->](0,0)--(1,-1);
\end{tikzpicture}
&
\begin{tikzpicture}[scale=.2, baseline=(current bounding box.center)]

\draw[thick,->](0,0)--(0,1);
\draw[thick,->](0,0)--(1,1);
\draw[thick,->](0,0)--(-1,0);
\draw[thick,->](0,0)--(-1,-1);
\draw[thick,->](0,0)--(1,-1);
\end{tikzpicture}
&
\begin{tikzpicture}[scale=.2, baseline=(current bounding box.center)]

\draw[thick,->](0,0)--(0,1);
\draw[thick,->](0,0)--(1,1);
\draw[thick,->](0,0)--(-1,-1);
\draw[thick,->](0,0)--(0,-1);
\draw[thick,->](0,0)--(1,-1);
\end{tikzpicture}
&
\begin{tikzpicture}[scale=.2, baseline=(current bounding box.center)]

\draw[thick,->](0,0)--(0,1);
\draw[thick,->](0,0)--(1,1);
\draw[thick,->](0,0)--(1,0);
\draw[thick,->](0,0)--(-1,-1);
\draw[thick,->](0,0)--(0,-1);
\draw[thick,->](0,0)--(1,-1);
\end{tikzpicture}
\\
\begin{tikzpicture}[scale=.2, baseline=(current bounding box.center)]

\draw[thick,->](0,0)--(0,1);
\draw[thick,->](0,0)--(1,1);
\draw[thick,->](0,0)--(-1,0);
\draw[thick,->](0,0)--(1,0);
\draw[thick,->](0,0)--(-1,-1);
\draw[thick,->](0,0)--(1,-1);
\end{tikzpicture}
&
 \begin{tikzpicture}[scale=.2, baseline=(current bounding box.center)]

\draw[thick,->](0,0)--(0,1);
\draw[thick,->](0,0)--(1,1);
\draw[thick,->](0,0)--(-1,0);
\draw[thick,->](0,0)--(-1,-1);
\draw[thick,->](0,0)--(0,-1);
\draw[thick,->](0,0)--(1,-1);
\end{tikzpicture}
&
\begin{tikzpicture}[scale=.2, baseline=(current bounding box.center)]

\draw[thick,->](0,0)--(0,1);
\draw[thick,->](0,0)--(-1,0);
\draw[thick,->](0,0)--(1,0);
\draw[thick,->](0,0)--(0,-1);
\draw[thick,->](0,0)--(1,-1);
\end{tikzpicture}
&

\begin{tikzpicture}[scale=.2, baseline=(current bounding box.center)]

\draw[thick,->](0,0)--(-1,1);
\draw[thick,->](0,0)--(0,1);
\draw[thick,->](0,0)--(1,0);
\draw[thick,->](0,0)--(-1,-1);
\draw[thick,->](0,0)--(1,-1);
\end{tikzpicture}
&
\begin{tikzpicture}[scale=.2, baseline=(current bounding box.center)]

\draw[thick,->](0,0)--(-1,1);
\draw[thick,->](0,0)--(0,1);
\draw[thick,->](0,0)--(-1,0);
\draw[thick,->](0,0)--(1,0);
\draw[thick,->](0,0)--(-1,-1);
\draw[thick,->](0,0)--(1,-1);
\end{tikzpicture}
&
\begin{tikzpicture}[scale=.2, baseline=(current bounding box.center)]

\draw[thick,->](0,0)--(-1,1);
\draw[thick,->](0,0)--(0,1);
\draw[thick,->](0,0)--(-1,0);
\draw[thick,->](0,0)--(1,0);
\draw[thick,->](0,0)--(-1,-1);
\draw[thick,->](0,0)--(0,-1);
\draw[thick,->](0,0)--(1,-1);
\end{tikzpicture}
&
\begin{tikzpicture}[scale=.2, baseline=(current bounding box.center)]

\draw[thick,->](0,0)--(0,1);
\draw[thick,->](0,0)--(-1,0);
\draw[thick,->](0,0)--(0,-1);
\draw[thick,->](0,0)--(1,-1);
\end{tikzpicture}
&
\begin{tikzpicture}[scale=.2, baseline=(current bounding box.center)]

\draw[thick,->](0,0)--(0,1);
\draw[thick,->](0,0)--(-1,0);
\draw[thick,->](0,0)--(-1,-1);
\draw[thick,->](0,0)--(1,-1);
\end{tikzpicture}
&
\begin{tikzpicture}[scale=.2, baseline=(current bounding box.center)]

\draw[thick,->](0,0)--(-1,1);
\draw[thick,->](0,0)--(0,1);
\draw[thick,->](0,0)--(-1,-1);
\draw[thick,->](0,0)--(1,-1);
\end{tikzpicture}
&
\begin{tikzpicture}[scale=.2, baseline=(current bounding box.center)]

\draw[thick,->](0,0)--(-1,1);
\draw[thick,->](0,0)--(0,1);
\draw[thick,->](0,0)--(-1,0);
\draw[thick,->](0,0)--(1,-1);
\end{tikzpicture}
&
\begin{tikzpicture}[scale=.2, baseline=(current bounding box.center)]

\draw[thick,->](0,0)--(0,1);
\draw[thick,->](0,0)--(-1,0);
\draw[thick,->](0,0)--(-1,-1);
\draw[thick,->](0,0)--(0,-1);
\draw[thick,->](0,0)--(1,-1);
\end{tikzpicture}
&
\begin{tikzpicture}[scale=.2, baseline=(current bounding box.center)]

\draw[thick,->](0,0)--(-1,1);
\draw[thick,->](0,0)--(0,1);
\draw[thick,->](0,0)--(-1,-1);
\draw[thick,->](0,0)--(0,-1);
\draw[thick,->](0,0)--(1,-1);
\end{tikzpicture}
&
 \begin{tikzpicture}[scale=.2, baseline=(current bounding box.center)]

\draw[thick,->](0,0)--(-1,1);
\draw[thick,->](0,0)--(0,1);
\draw[thick,->](0,0)--(-1,0);
\draw[thick,->](0,0)--(0,-1);
\draw[thick,->](0,0)--(1,-1);
\end{tikzpicture}
&
 \begin{tikzpicture}[scale=.2, baseline=(current bounding box.center)]

\draw[thick,->](0,0)--(-1,1);
\draw[thick,->](0,0)--(0,1);
\draw[thick,->](0,0)--(-1,0);
\draw[thick,->](0,0)--(-1,-1);
\draw[thick,->](0,0)--(1,-1);
\end{tikzpicture}
&
\begin{tikzpicture}[scale=.2, baseline=(current bounding box.center)]

\draw[thick,->](0,0)--(-1,1);
\draw[thick,->](0,0)--(0,1);
\draw[thick,->](0,0)--(-1,0);
\draw[thick,->](0,0)--(-1,-1);
\draw[thick,->](0,0)--(0,-1);
\draw[thick,->](0,0)--(1,-1);
\end{tikzpicture}
&

\end{array}$

\end{trivlist}
\caption{Current classification of small step models of walks in the three quarter plane. Except for the 5 rational cases, the D-finite cases correspond to walks with a finite group. The red framed models are of different nature in the quarter and the three quarter plane.}\label{figcas34}
\end{figure}

\vspace*{0.5cm}
\paragraph{\textbf{Method and main results.}}
As mentioned in \cite{RaTr-18, Tr-19}, the study of walks avoiding a quadrant gives rise to convergence problems. Indeed, although we can easily write a functional equation for the generating series $C(x,y;t)$, unlike the quadrant case, see \cite[Sec.~4.1]{BMM}, this functional equation involves infinitely many negative and positive powers of $x$ and $y$ making the series not convergent anymore. In this article, we follow the same strategy as \cite{RaTr-18, Tr-19} and divide the three quadrants into two symmetric convex cones of opening angle $3\pi/4$ and a diagonal in between, see Figure~\ref{fig:three-quarter_split}. After making some assumptions on the step set of the walk (diagonal symmetry and no anti-diagonal jumps in \ref{main_hyp}), we derive a functional equation which can be solved. We are then able to use the tools and techniques of \cite{DHRS} to find the nature of the generating series for the models of   Figure~\ref{fig:symmetric_models_infinite}. More precisely, we prove the following, see Theorem \ref{thm2}.
\begin{figure}[t]
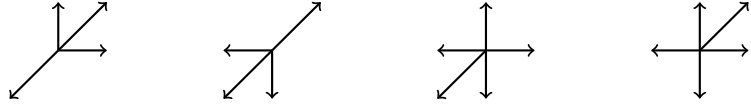

	\centering
	\begin{tabular}{cccc}
		$\ \diagr{E,NE,N,SW}\ $
		\qquad&\qquad 
		$\ \diagr{W,NE,S,SW}\ $ 
		\qquad&\qquad
		$\ \diagr{E,W,N,S,SW}\ $
		\qquad&\qquad
		$\ \diagr{E,W,N,S,NE}\ $ 
	\end{tabular}
	\caption{Diagonally symmetric step sets with infinite group and no anti-diagonal directions.}
	\label{fig:symmetric_models_infinite}
\end{figure}

\begin{thmintro}
 The following holds:
\begin{itemize}
\item Assume that the step set of the walk is the following: 
$$\begin{tikzpicture}[scale=.5, baseline=(current bounding box.center)]
\draw[thick,->](0,0)--(0,1);
\draw[thick,->](0,0)--(0,-1);
\draw[thick,->](0,0)--(1,0);
\draw[thick,->](0,0)--(-1,0);
\draw[thick,->](0,0)--(1,1);
\end{tikzpicture}$$
Then, $C(x,y;t)$ is D-algebraic. 
\item Assume that the step set of the walk is one of the following:
$$
\begin{tikzpicture}[scale=.5, baseline=(current bounding box.center)]
\draw[thick,->](0,0)--(0,1);
\draw[thick,->](0,0)--(1,1);
\draw[thick,->](0,0)--(1,0);
\draw[thick,->](0,0)--(-1,-1);
\end{tikzpicture}\quad  \quad \begin{tikzpicture}[scale=.5, baseline=(current bounding box.center)]
\draw[thick,->](0,0)--(0,-1);
\draw[thick,->](0,0)--(1,1);
\draw[thick,->](0,0)--(-1,0);
\draw[thick,->](0,0)--(-1,-1);
\end{tikzpicture}\quad \quad \begin{tikzpicture}[scale=.5, baseline=(current bounding box.center)]
\draw[thick,->](0,0)--(0,1);
\draw[thick,->](0,0)--(0,-1);
\draw[thick,->](0,0)--(1,0);
\draw[thick,->](0,0)--(-1,0);
\draw[thick,->](0,0)--(-1,-1);
\end{tikzpicture}$$
Then, $ C(x,y;t)$ is D-transcendental. More precisely, let $|\mathcal{S}|$ denote the cardinality of $\mathcal{S}$. There exists $t\in(0,1/|\mathcal{S}|)$ such that the series $x\mapsto C(x,y;t)$ (resp.~$y\mapsto C(x,y;t)$) does not satisfy any nontrivial algebraic differential equation with coefficients in $\Q(x,y,t)$.
\end{itemize}
\end{thmintro}

We emphasize  the fact that for the four models of Figure \ref{fig:symmetric_models_infinite}, the generating series is of the same nature in the quarter and the three quarter plane. In the same vein as the remarks of Bousquet-M\'elou in~\cite[Sec.~7.2]{BM-16}, we make the following conjecture
\begin{conj}
Suppose that the step set $\mathcal S$ is not included in any half-space with boundary passing through the origin. The generating series of walks in the three quarter plane with step set $\mathcal{S}$ is of the same nature as the generating series of walks in the quarter plane with step set $\mathcal{S}$. 
\end{conj}
Note that for all the few models that have already been solved, the methods used in the quarter plane and the three quarter plane are similar, yet much more difficult to implement in three quarter plane case.

\vspace*{0.5cm}
\paragraph{\textbf{Structure of the paper.}} 
In Section~\ref{sec:Kernel-FctEq}, we start with the introduction of an important object called the kernel of the walk denoted by $K(x,y)$.  We further derive various functional equations satisfied by the generating series of the walks.  Finally, after presenting some properties of the kernel curve, we recall the notion of the group of the walk.  This group only depends on the steps of the walk and is generated by two bi-rational transformations. 
In Section~\ref{sec:AnalyticContinuation}, we first uniformize the elliptic curve defined by the zeros of the kernel. We analytically continue the generating series on the kernel curve and then, after parametrization, on the whole complex plane $\C$. 
In Section~\ref{sec:DiffTranscendance}, we study the nature of generating series corresponding to the four step sets of Figure \ref{fig:symmetric_models_infinite}. \\\par

\textbf{Acknowledgment.} The authors want to warmly thank the referees for their  detailed and helpful comments. The authors also thank Samuel Simon for his help with the English language, Kilian Raschel, and Manuel Kauers,  for their suggestions.

\section{Kernel and functional equations}
\label{sec:Kernel-FctEq}

\subsection{Weighted walks}
In this paper, we are going to consider the more general context of weighted walks.
More precisely, let $(d_{i,j})_{(i,j)\in\{0,\pm 1\}^{2}}$ be a family of elements of $\Q\cap [0,1]$ such that $\sum_{i,j} d_{i,j}=1$.    We  consider weighted models of  walks in the three quarter plane $\mathcal C$ satisfying the following properties: 
\begin{itemize} 
\item the walks start at $(0,0)$; 
\item for $(i,j)\in\{0,\pm 1\}^{2}\backslash\{(0,0)\}$ (resp.~$(0,0)$), the  element   $d_{i,j}$ is  a  weight of the step $(i,j)$  and can be viewed as  the probability for the walk  to go in the direction $(i,j)$ (resp.~to stay  at the same position).
\end{itemize}

The set of steps  of the weighted walk  is  the 
set of nearest neighbors with nonzero weight, that is, 
\[
\mathcal{S}=\{(i,j) \in\{0,\pm 1\}^{2} | d_{i,j} \neq 0 \}.
\]

If $d_{0,0}=0$ and if the nonzero $d_{i,j}$ all have the same value $1/|\mathcal{S}|$, we say the model is unweighted. \par 
The {\it weight of the walk} is defined as the product of the weights of its component steps.
From now on, we replace the definition of $c_{i,j}(n)$ by the {sum of the weights of all walks reaching}  the position $(i,j)\in \mathcal{C}$ from the initial position $(0,0)$ after $n$ steps. We also replace 
 $C(x,y;t)$ by the following trivariate generating series
$$
C(x,y;t):=\displaystyle \sum_{n=0}^{\infty}\sum_{(i,j)\in \mathcal{C}}c_{i,j}(n)x^{i}y^{j}t^{n}.
$$

\begin{rem}
Consider an unweighted walk with generating series $C(x,y;t)$.
Then, the series $C(x,y;t|\mathcal{S}|)$, represents the generating series in the introduction. We stress  the fact that the latter has the same nature as $C(x,y;t)$, so  this abuse of notation has no importance in what follows. \end{rem}

\begin{rem}For simplicity, we assume that the weights $d_{i,j}$ belong to  $\Q$. However, we would like to mention that any of the arguments and statements  below will hold with arbitrary real weights in $[0,1]$. One just needs to  replace the field $\Q$ with the field $\Q(d_{i,j})$. \end{rem}

From now on, we fix $0<t<1$. To simplify notation, we may sometimes omit the dependencies in $t$. In particular,  $C(x,y;t)$ will also be denoted by $C(x,y)$.

\subsection{Kernel of the walk}

Usually, the starting point is to reduce the enumeration problem of walks to the resolution of a functional equation. Before deriving this equation satisfied by the generating series $C(x,y)$, let us present the kernel, a crucial object, which depends on the step set.

The polynomial
\begin{equation} 
	\label{eq:Kernel}
	K(x,y)=xy\left[t\sum_{(i,j)\in\mathcal{S}}d_{i,j}x^{i}y^{j}-1 \right]
\end{equation}
is called the kernel\index{kernel} of the walk. It encodes the elements of $\mathcal{S}$ (the steps of the walk). We can rewrite it as:
\begin{equation}
	\label{eq:kernel_expanded}
	K(x,y) = a(x)y^2 + b(x)y + c(x)= \widehat{a}(y)x^2 + \widehat{b}(y)x + \widehat{c}(y).
\end{equation}

We also define the discriminants in $x$ and $y$ of the kernel \eqref{eq:Kernel}:
\begin{equation}
	\label{eq:discriminants}
	d(x) = b(x)^2 - 4a(x)c(x) \qquad \text{and}\qquad \widehat{d}(y) = \widehat{b}(y)^2 - 4\widehat{a}(y)\widehat{c}(y).
\end{equation}

Let us fix once for all a determination of the complex logarithm $\log$,  and let us consider the following determination of the complex square root  $\sqrt{z}=\exp(\log(z)/2)$.  Let $Y_{\pm} (x)$ (resp.~$X_{\pm}(y)$) be the algebraic functions defined by the relation $K(x,Y_{\pm}(x))=0$ (resp.~$K(X_{\pm}(y),y)=0$). Obviously by \eqref{eq:kernel_expanded} and \eqref{eq:discriminants} we have
\begin{equation}
	\label{eq:algebraic_expressions_Y_X}
	Y_{\pm}(x)=\frac{-b(x)\pm\sqrt{d(x)}}{2a(x)}\qquad \text{and}\qquad X_{\pm}(y)=\frac{-\widehat{b}(y)\pm\sqrt{\widehat{d}(y)}}{2\widehat{a}(y)}.
\end{equation}

\subsection{Functional equations}

From now on, we impose the following assumption:

\begin{enumerate}[label=\textcolor{red}{\textbf{(A1)}},ref={\rm (A1)}]
     \item\label{main_hyp} We forbid anti-diagonal steps: for all $(i,j)\in \mathcal{S}$, $d_{i,j}=d_{j,i}$, and $d_{1,-1}=d_{-1,1}=0$.
\end{enumerate} 
As in the case of the generating series of quadrant walks \cite[Sec.~4.1]{BMM}, using a recursive construction of a walk by adding a new step at the end of the walk at each stage, we can easily derive a first functional equation for $C(x,y)$, see \cite[Sec.~2]{BM-16} and \cite[Sec.~2.1]{RaTr-18}. We obtain

\begin{multline}
	\label{eq:functional_equation_3/4}
C(x,y)=1+t\sum_{(i,j)\in\mathcal S}d_{i,j}x^{i}y^{j}C(x,y)-t\sum_{i\in \{0,\pm 1\}}d_{i,-1}x^{i}y^{-1}C_{-0}(x^{-1})
\\-t\sum_{j\in \{0,\pm 1\}}d_{-1,j}x^{-1}y^{j}C_{0-}(y^{-1})
	-td_{-1,-1}x^{-1}y^{-1}C_{0,0},
\end{multline}
where
\begin{align*}
     C_{-0}(x^{-1})&= \sum_{\substack{n\geq0 \\i  > 0}} c_{-i,0}(n)x^{-i}t^{n},\\
     C_{0-}(y^{-1}) &= \sum_{\substack{n\geq0 \\j  > 0}} c_{0,-j}(n)y^{-j}t^{n},\\
     C_{0,0}&=\sum_{n\geq0}c_{0,0}(n)t^{n}.
\end{align*}
Multiplying both sides by $-xy$ we get 
\begin{equation}\label{eq1}
K(x,y)C(x,y)=c(x)C_{-0}(x^{-1})+\widehat{c}(y)C_{0-}(y^{-1})+td_{-1,-1}C_{0,0}-xy.
\end{equation}
Because of the presence of infinitely many terms with positive and negative powers of $x$ and $y$, the resolution of the functional equation \eqref{eq:functional_equation_3/4} by algebraic substitutions or evaluations at well-chosen complex points seems complicated, as the series are no longer convergent. Therefore, we use the same strategy as in \cite{RaTr-18}: we split the three quadrant into three parts and decompose the generating series $C(x,y)$ into a sum of three generating series

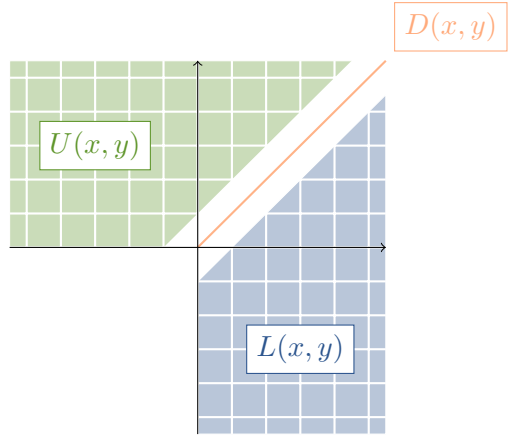
\begin{figure}[t]
	\centering
	\begin{tikzpicture}
		\begin{scope}[scale=0.45]
			\tikzstyle{quadri}=[rectangle,draw,fill=white]
			\draw[white, fill=dblue!30] (0,-1) -- (5.5,4.5) -- (5.5,-5.5) -- (0,-5.5);
			\draw[white, fill=dgreen!30] (-1,0) -- (-5.5,0) -- (-5.5,5.5) -- (4.5,5.5);
			\draw[white, thick] (0,-5.5) grid (5.5,5.5);
			\draw[white, thick] (-5.5,0) grid (0,5.5);
			\draw[Apricot!90, thick] (0,0) -- (5.5,5.5);
			\draw[->] (0,-5.5) -- (0,5.5);
			\draw[->] (-5.5,0) -- (5.5,0);
			\node[dblue!90,quadri] at (3,-3) {${L}(x,y)$};
		\node[dgreen!90, quadri] at (-3,3) {${U}(x,y)$};
		\node[Apricot!100, quadri] at (7.4,6.5) {${D}(x,y)$};
		\end{scope}
	\end{tikzpicture}
	\caption{Decomposition of the three quadrant cone into two cones of opening angle $\frac{3\pi}{4}$. The green part corresponds to $U$, the blue part to $L$, and the orange part to $D$.}
	\label{fig:three-quarter_split}
\end{figure}

\begin{equation}
	\label{eq:equation_cut3parts}
	C(x,y)=L(x,y)+D(x,y)+U(x,y),
\end{equation}
where 
\begin{equation*}
	L(x,y)=\sum\limits_{\substack{i \geq 0 \\  j\leq i-1 \\n\geq 0}}c_{i,j}(n)x^i y^j t^n, \quad
	D(x,y)=\sum\limits_{\substack{i\geq 0 \\ n\geq 0}}c_{i,i}(n)x^i y^i t^n, \quad
	U(x,y)=\sum\limits_{\substack{j \geq 0 \\ i\leq j-1 \\n\geq 0}}c_{i,j}(n)x^i y^j t^n.
\end{equation*}
The generating series $L(x,y)$ (resp.~$D(x,y)$, and $U(x,y)$) represents walks ending in the lower cone $\{ i\geq 0, j\leq i-1\}$ (resp.~the diagonal $\{i=j\}$, and the upper cone $\{j\geq 0, i\leq j-1\}$), see Figure~\ref{fig:three-quarter_split}.\par 

Thanks to the diagonal symmetry of the step set and the fact that the starting point $(0,0)$ lies on the diagonal, we have $U(x, y) = L(y, x)$ and $C(x, y)$ can be written as the sum $L(x, y)+D(x, y)+L(y, x)$ of two unknown generating series. 
\begin{lem}
	\label{lem:functional_equation_sym}
	For any step set which satisfies \ref{main_hyp} and any walk that starts at $(0,0)$, one has 
\begin{multline}\label{eq:functional_equation_sym}
		K(x,y)L(x,y)=-\frac{xy}{2}+\widehat{c}(y)C_{0-}(y^{-1}) + \frac{t}{2}d_{-1,-1}C_{0,0}
		\\
		-xy\left(-\frac{1}{2}+t\left(\frac{1}{2}(d_{1,1}xy+d_{0,0}+d_{-1,-1}x^{-1}y^{-1})+ d_{0,-1}y^{-1} + d_{1,0}x \right) \right)D(x,y).
	\end{multline}
\end{lem}

\begin{proof}
This is basically \cite[Lem.~1]{RaTr-18} and may be straightforwadly deduced from the proof of \cite[Appx.~C]{RaTr-18}.
\end{proof}

As in \cite{RaTr-18}, let us perform on \eqref{eq:functional_equation_sym} the change of coordinates 
\begin{equation}\label{eq:change_var}
	\varphi(x,y)=(xy, x^{-1}).
\end{equation}
Note that it is bijective with inverse $\varphi^{-1}(x,y)=(y^{-1},xy)$. This change of coordinates $\varphi$ acts on monomials and the plane is transformed via $(k,l)\rightarrow (k-l,k)$. Hence $\varphi$ changes $L$ into the positive quadrant (see Figure~\ref{fig:half-plane_split}).
If we also multiply the functional equation~\eqref{eq:functional_equation_sym} by $x$, we get by \cite[(15)]{RaTr-18}, the functional equation

	\begin{equation}
	\label{eq:functional_equation_octant}
	K_{\varphi}(x,y)L_{\varphi}(x,y)
	=xf(x)C_{0-}(x)+xg(x,y)D_{\varphi}(y)+\frac{t}{2}d_{-1,-1}xC_{0,0}-\frac{xy}{2},
\end{equation}
where 
$$\begin{array}{lll}
K_{\varphi}(x,y)& =&xK(\varphi(x,y))=xy\left(t\displaystyle\sum_{(i,j)\in \mathcal{S}}d_{i,j}x^{i-j}y^{i}-1\right)=xy\left(t\displaystyle\sum_{(i,j)\in \{ 0,\pm 1\}^{2}}d_{j,j-i}x^{i}y^{j}-1\right),\\
f(x)&= &td_{-1,0}+txd_{-1,-1},\\
 g(x,y)&=&-y\Big(-\frac{1}{2}+t\Big(\frac{1}{2}(d_{1,1}y+d_{0,0}+d_{-1,-1}y^{-1})+ d_{0,-1}x + d_{1,0}xy \Big) \Big),\\
 L_{\varphi}(x,y)& = &L(\varphi(x,y)) = \sum\limits_{\substack{i \geq 1 \\  j\geq 0 \\n\geq 0}}c_{j,j-i}(n)x^i y^j t^n,\\
C_{0-}(x)&=&\sum\limits_{\substack{j> 0 \\ n\geq 0}} c_{0,-j}(n)x^{j}t^{n},\\
 D_{\varphi}(y) &=& D(\varphi(x,y)) = \sum\limits_{\substack{i\geq 0 \\ n\geq 0}} c_{i,i}(n)y^i t^n.
\end{array}$$

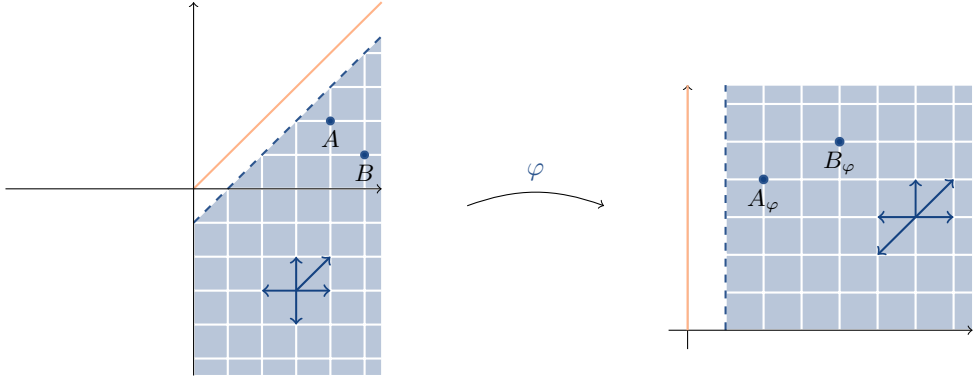
\begin{figure}[t]
	\centering
	\begin{tikzpicture}
		\begin{scope}[scale=0.45, xshift=0cm, yshift=2.5cm]
			\draw[white, fill=dblue!30] (0,-1) -- (5.5,4.5) -- (5.5,-5.5) -- (0,-5.5);
			\draw[white, thick] (0,-5.5) grid (5.5,5.5);
			\draw[white, thick] (-5.5,0) grid (0,5.5);
			\draw[Apricot!90, thick] (0,0) -- (5.5,5.5);
			\draw[thick, dblue!90, fill=dblue] (5,1) circle (0.10 cm);
			\node[below] at (5,1) {\footnotesize{$B$}};
			\draw[thick, dblue!90, fill=dblue] (4,2) circle (0.10 cm);
			\node[below] at (4,2) {\footnotesize {$A$}};
			\draw[dashed, dblue!90, thick] (0,-1) -- (5.5,4.5);
			\draw[->] (0,-5.5) -- (0,5.5);
			\draw[->] (-5.5,0) -- (5.5,0);
			\draw[thick, dblue, ->] (3,-3) -- (2,-3);
			\draw[thick, dblue, ->] (3,-3) -- (4,-3);
			\draw[thick, dblue, ->] (3,-3) -- (3,-4);
			\draw[thick, dblue, ->] (3,-3) -- (4,-2);
			\draw[thick, dblue, ->] (3,-3) -- (3,-2);
		\end{scope}

		\begin{scope}[scale=0.3, xshift= 15cm]
			\draw[->] (-3,3) to[out=20,in=160] (3,3);
			\draw [] (0,4.5) node[dblue!90]{\small {$\varphi$}};
		\end{scope}		

		\begin{scope}[scale=0.5, xshift=13cm, yshift=-1.5cm]
			\draw[white, fill=dblue!30] (1,0) -- (1,6.5) -- (7.5,6.5) -- (7.5,0);
			\draw[white, thick] (0,0) grid (7.5,6.5);
			\draw[dashed, dblue!90, thick] (1,0) -- (1,6.5);
			\draw[->] (0,-0.5) -- (0,6.5);
			\draw[->] (-0.5,0) -- (7.5,0);
			\draw[Apricot!90, thick] (0,0) -- (0,6.5);
			\draw[thick, dblue, ->] (6,3) -- (5,3);
			\draw[thick, dblue, ->] (6,3) -- (6,4);
			\draw[thick, dblue, ->] (6,3) -- (7,4);
			\draw[thick, dblue, ->] (6,3) -- (7,3);
			\draw[thick, dblue, ->] (6,3) -- (5,2);
			\draw[thick, dblue!90, fill=dblue] (2,4) circle (0.10 cm);
			\node[below] at (2,4) {\footnotesize {$A_{\varphi}$}};
			\draw[thick, dblue!90, fill=dblue] (4,5) circle (0.10 cm);
			\node[below] at (4,5) {\footnotesize{$B_{\varphi}$}};
		\end{scope}
	\end{tikzpicture}
	\caption{After performing the change of coordinates $\varphi$ defined in 
	\eqref{eq:change_var},
	the lower octant is transformed into the right quadrant. The weights of the walk are also changed by $\varphi$, see Lemma~\ref{lem1}.}
	\label{fig:half-plane_split}
\end{figure}
 Note that $f$ and $g$ are both polynomials. The multiplication by $x$ in  \eqref{eq:functional_equation_octant} permits us to find a kernel $K_{\varphi}(x,y)$ which is of the same form as the original kernel $K(x,y)$. More precisely, the following lemma holds and is a straightforward consequence of the expression of $K_{\varphi}(x,y)$.
 
 \begin{lem}\label{lem1}
The polynomial $K_{\varphi}(x,y)$ corresponds to a kernel of a weighted walk with weights $d^{\varphi}_{i,j}$, where 
$$\begin{array}{lllllllll}
d^{\varphi}_{-1,1}&=&0,&d^{\varphi}_{0,1}&=&d_{1,1},&d^{\varphi}_{1,1}&=&d_{1,0},\\
d^{\varphi}_{-1,0}&=&d_{0,1},&d^{\varphi}_{0,0}&=&d_{0,0},&d^{\varphi}_{1,0}&=&d_{0,-1},\\
d^{\varphi}_{-1,-1}&=&d_{-1,0},&d^{\varphi}_{0,-1}&=&d_{-1,-1},&d^{\varphi}_{1,-1}&=&0.
\end{array}$$
 \end{lem}
\begin{rem}
The second part of Assumption~\ref{main_hyp} $d_{1,-1}=d_{-1,1}=0$ avoids the appearance of large steps after the change of variable $\varphi$ defined in~\eqref{eq:change_var}. For example, $d_{1,-1}$ would be changed into $d^{\varphi}_{2,1}$, see Figure~\ref{fig:half-plane_split}.
\end{rem}

By Lemma \ref{lem1}, $K_{\varphi}(x,y)$ has degree at most two in $x$ and degree at most two in $y$. We may then write 
$$
K_{\varphi}(x,y) = a_{\varphi}(x)y^2 + b_{\varphi}(x)y + c_{\varphi}(x) =\widehat{a}_{\varphi}(y)x^2 + \widehat{b}_{\varphi}(y)x + \widehat{c}_{\varphi}(y).
$$
When $K_{\varphi}(x,y)$ has exactly degree two in $x$ and $y$, we also define the discriminants in $x$ and $y$ of the kernel $K_{\varphi}(x,y)$:
$$
d_{\varphi}(x) = b_{\varphi}(x)^2 - 4a_{\varphi}(x)c_{\varphi}(x)	 \qquad \text{and}\qquad \widehat{d}_{\varphi}(y) = \widehat{b}_{\varphi}(y)^2 - 4\widehat{a}_{\varphi}(y)\widehat{c}_{\varphi}(y).
$$

We then have  $K_{\varphi}(x,Y_{\varphi \pm}(x))=0$ (resp.~$K_{\varphi}(X_{\varphi \pm}(y),y)=0$) where
$$
	Y_{\varphi \pm}(x)=\frac{-b_{\varphi}(x)\pm\sqrt{d_{\varphi}(x)}}{2a_{\varphi}(x)}\qquad \text{and}\qquad X_{\varphi \pm}(y)=\frac{-\widehat{b}_{\varphi}(y)\pm\sqrt{\widehat{d}_{\varphi}(y)}}{2\widehat{a}_{\varphi}(y)}.
$$
\subsection{The kernel curve}

 Let us  consider the algebraic curve  $$
 E = \{(x,y) \in \C \times \C \ \vert \ K(x,y) = 0\}. 
 $$ 
As it is not compact, we need to embed $E$ into $\P1(\C) \times \P1(\C)$. 
 We first recall that $\P1(\C)$ denotes the complex projective line, which is the quotient of $\C \times \C \setminus \{(0,0)\}$ by the equivalence relation $\sim$ defined by 
$$
(x_{0},x_{1}) \sim (x_{0}',x_{1}') \Leftrightarrow \exists \lambda \in \C^{*},  (x_{0}',x_{1}') = \lambda (x_{0},x_{1}). 
$$
The equivalence class of $(x_{0},x_{1}) \in \C \times \C \setminus \{(0,0)\}$ is usually denoted by $[x_{0}:x_{1}] \in \P1(\C)$. The map 
$
x \mapsto  [x:1]
$ 
embeds $\C$ inside $\P1(\C)$. The latter map is not surjective: its image is $\P1(\C) \setminus \{[1:0]\}$; the missing point $[1:0]$  is usually denoted by $\infty$. 
  Now, we embed $E$  inside $\P1(\C) \times \P1(\C)$ via  ${(x,y) \mapsto ([x:1],[y:1])}$. The kernel curve $\Etproj$ is the closure of this embedding of $E$.  In other words, $\Etproj $ is the algebraic curve defined by 
$$
\Etproj = \{([x_{0}:x_{1}],[y_{0}:y_{1}]) \in \P1(\C) \times \P1(\C) \ \vert \ \overline{K}(x_0,x_1,y_0,y_1) = 0\}
$$
where $\overline{K}(x_0,x_1,y_0,y_1)$ is the following bihomogeneous polynomial
\begin{equation}\label{eq:kernelwalk}
\overline{K}(x_0,x_1,y_0,y_1)={x_1^2y_1^2K\left(\frac{x_0}{x_1},\frac{y_0}{y_1}\right)}= t \sum_{i,j=0}^2 d_{i-1,j-1} x_0^{i} x_1^{2-i}y_0^j y_1^{2-j}-x_0x_1y_0y_1. 
 \end{equation}

The following proposition has been proved in \cite[Prop.~2.1 and Cor.~2.6]{DHRS20}, when $t$ is transcendental and has been extended for a general $0<t<1$ in \cite[Prop.~1.9]{dreyfus2017differential}.
\begin{prop}\label{prop2}
The following facts are equivalent.
\begin{enumerate}
\item $\Etproj$ is an elliptic curve;
\item  The set of authorized directions $\mathcal{S}$ is not included in any half space with boundary passing through the origin.
\end{enumerate}
\end{prop}

From now on, we make the following additional assumption.

\begin{enumerate}[label=\textcolor{red}{\textbf{(A2)}},ref={\rm (A2)}]
     \item\label{main_hyp2} The set of authorized directions $\mathcal{S}$ is not included in any half space with boundary passing through the origin.
\end{enumerate}

Let 
$$ \overline{K}_{\varphi}(x_0,x_1,y_0,y_1)={x_1^{2} y_1^2 K_{\varphi}\left(\frac{x_0}{x_1},\frac{y_0}{y_1}\right)}= t \sum_{i,j=0}^2 d^{\varphi}_{i-1,j-1} x_0^{i} x_1^{2-i}y_0^j y_1^{2-j}-x_0x_1y_0y_1 , $$
and define $$
\Etproj_{\varphi} = \{([x_{0}:x_{1}],[y_{0}:y_{1}]) \in \P1(\C) \times \P1(\C) \ \vert \ \overline{K}_{\varphi}(x_0,x_1,y_0,y_1) = 0\}.
$$

\begin{prop}\label{prop1}
Suppose that Assumptions \ref{main_hyp} and \ref{main_hyp2} hold.  Then, 
$\Etproj_{\varphi}$ is an elliptic curve.
\end{prop}

\begin{proof}
 By Lemma \ref{lem1} and the fact that Assumptions \ref{main_hyp}, \ref{main_hyp2} hold, the set of authorized directions of $K_{\varphi}(x,y)$ is not included in any half space with boundary passing through the origin. This follows now from Proposition~\ref{prop2}.
\end{proof}

\subsection{The Group of the walk}

Recall that we have seen in Lemma~\ref{lem1}, that  $K_{\varphi}(x,y)$ has degree  at most two in $x$ and $y$ and nonzero coefficient of degree $0$ in $x$ and $y$. Hence, $a_{\varphi}(x),c_{\varphi}(x),\widehat{a}_{\varphi}(y),\widehat{c}_{\varphi}(y)$ are not identically zero. \par 
In what follows, we use the classical dashed arrow notation to denote rational maps; {\it a priori}, such functions may not be defined everywhere. Following \cite[Sec.~3]{BMM}, \cite[Sec.~3]{KauersYatchak},  or \cite{FIM}, we consider the involutive rational functions
$$ 
i_{1}, i_{2} : \P1(\C) \times \P1(\C) \dashrightarrow \P1(\C) \times \P1(\C)
$$ 
 given by 
$$
i_1([x_0:x_1]),[y_0:y_1])) =\left(\frac{x_{0}}{x_{1}}, \frac{c_{\varphi}(\frac{x_{0}}{x_{1}}) }{a_{\varphi}(\frac{x_{0}}{x_{1}})\frac{y_{0}}{y_{1}}}\right) \quad \text{and} \quad   i_2([x_0:x_1]),[y_0:y_1]))=\left(\frac{\widehat{c}_{\varphi}(\frac{y_{0}}{y_{1}})}{\widehat{a}_{\varphi}(\frac{y_{0}}{y_{1}})\frac{x_{0}}{x_{1}}},\frac{y_{0}}{y_{1}}\right).
$$ 
Note that we have $i_1([x_0/x_1:1],[y_0/y_1:1])=i_1([x_0:x_1],[y_0:y_1])$ and the same holds for $i_2$.
Note that $i_{1}, i_{2}$  are {\it a priori} not defined when the denominators vanish but we will see in the sequel that we may overcome this problem when we will restrict to $\Etproj_{\varphi}$. \par 
For a fixed value of $x$, there are at most  two possible values of $y$ such that $(x,y)\in\Etproj_{\varphi}$. The involution $i_{1}$ corresponds to interchanging these values. A similar interpretation can be given for $i_2$. Therefore  the kernel curve $\Etproj_{\varphi}$ is left invariant by the natural action of $i_{1}, i_{2}$. \par 
 We denote by $\iota_{1}, \iota_{2}$ the restriction of $i_{1},i_{2}$ to $\Etproj_{\varphi}$, see Figure \ref{figiota}. Again, these functions are {\it a priori} not defined where the denominators vanish. However, by \cite[Prop.~3.1]{DHRS20}, this is only an ``apparent problem''. To be precise, they proved this for $t$ transcendental but the proofs stay valid for every $0<t<1$. We then obtain that
 $\iota_{1}$ and $\iota_{2}$ can be extended to morphisms of $\Etproj_{\varphi}$. We recall that a rational map $f :\Etproj_{\varphi} \dashrightarrow \Etproj_{\varphi}$ is a morphism if it is regular at any $P \in \Etproj_{\varphi}$, {\it i.e.}, if $f$ can be represented in suitable affine charts containing $P$ and  $f(P)$ by a rational function with nonvanishing denominator at $P$. \par

    



\begin{figure}
\begin{center}
\includegraphics{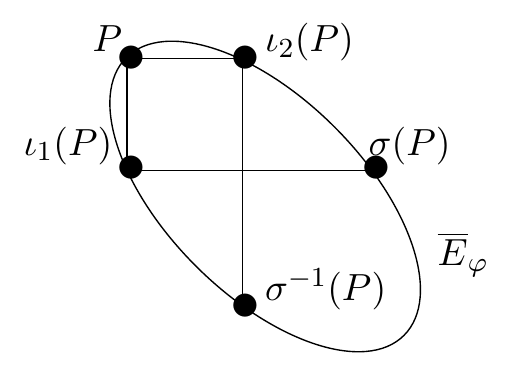}
\end{center}
\caption{The maps $i_{1}$ and $i_{2}$ restricted to the kernel curve $\Etproj_{\varphi}$ are denoted by $\iota_1$ and $\iota_2$, respectively.}\label{figiota}
\end{figure}

Let us finally define
\begin{equation*}
     \sigma=\iota_2 \circ \iota_1. 
\end{equation*}
Note that such a map is called a QRT-map and has been widely studied, see \cite{DuistQRT}. In the quarter plane case,  the algebraic nature of the generating  series depends in part on whether $\sigma$ has finite or infinite order. More precisely, in the unweighted quarter plane case, the group associated to the curve $\Etproj$ is finite if and only if the generating series is D-finite (see the introduction). In the unweighted three quarter plane, when $\Etproj$ is an elliptic curve, an infinite group associated to the curve $\Etproj$ implies that the generating series is not D-finite, see \cite[Thm.~1.3]{mustapha2019non}. Note that when the $d_{i,j}$ are fixed, the cardinality of the group depends on $t$, see \cite{fayolleRaschel} for concrete examples.\par 

Recall that the rational maps $f(x)$, $g(x,y)$ were defined in \eqref{eq:functional_equation_octant} and we may extend their definition to the elliptic curve $\Etproj_{\varphi}$, in the same way as the rational maps $\iota_1$ and $\iota_2$.
If no confusion is likely to  arise we are going to write $(x,y)$ instead of $([x_0:x_1],[y_0:y_1])$ for an element of $\P1(\C) \times \P1(\C)$. 
 We may now prove that on  the curve $\Etproj_{\varphi}$, $g(x,y)$ is almost independent of $x$. More precisely, the following lemma holds. 

\begin{lem}\label{lem3}
There exists $\varepsilon \in \{-1,1\}$, such that 
if we evaluate $g(x,y)$ on $\Etproj_{\varphi}$, then $g(x,y)= \frac{\varepsilon\sqrt{	\widehat{d}_{\varphi}(y)}}{2}$. 
Furthermore, for every $(x,y)\in\Etproj_{\varphi}$, $$g(\iota_{2}(x,y))=-g(x,y).$$
\end{lem}

\begin{proof}
Let $(x,y)\in \Etproj_{\varphi}$. We find $x=\frac{-\widehat{b}_{\varphi}(y)-\varepsilon\sqrt{	\widehat{d}_{\varphi}(y)}}{2\widehat{a}_{\varphi}(y)}$ for some $\varepsilon \in \{-1,1\}$.
Since $g(x,y)=-x\widehat{a}_{\varphi}(y)-\frac{\widehat{b}_{\varphi}(y)}{2}$,  it holds that
$$g(x,y)=-\frac{-\widehat{b}_{\varphi}(y)
-\varepsilon\sqrt{	\widehat{d}_{\varphi}(y)}}{2\widehat{a}_{\varphi}(y)}\times\widehat{a}_{\varphi}(y)-\frac{\widehat{b}_{\varphi}(y)}{2}=\frac{\varepsilon\sqrt{	\widehat{d}_{\varphi}(y)}}{2}.$$
Furthermore, the first coordinate of $\iota_{2}(x,y)$, which is, $\frac{-\widehat{b}_{\varphi}(y)+\varepsilon\sqrt{	\widehat{d}_{\varphi}(y)}}{2\widehat{a}_{\varphi}(y)}$, is the other root, and the same computation shows 
$$g(\iota_{2}(x,y))=-\frac{\varepsilon\sqrt{\widehat{d}_{\varphi}(y)}}{2}=-g(x,y).$$
\end{proof}
\section{Analytic continuation}
\label{sec:AnalyticContinuation}

\subsection{Uniformization of the kernel curve}\label{sec21}

Since by Proposition \ref{prop1}, $\Etproj_{\varphi}$ is an elliptic curve, we may identify $\Etproj_{\varphi}$ with $\C/(\Z\omega_1 + \Z\omega_2)$, where $(\omega_1,\omega_2)\in \C^{2}$ is the basis of a lattice, via the $(\Z\omega_1 + \Z\omega_2)$-periodic map  
$$
\Lambda : \left\{
\begin{array}{lll}
\C& \rightarrow &\Etproj_{\varphi} \\
\omega &\mapsto& (\mathfrak{q}_1(\omega), \mathfrak{q}_2(\omega)),
\end{array}
\right.
$$
where $\mathfrak{q}_1, \mathfrak{q}_2$ are rational functions of $\wp$ and its derivative $d\wp/d\omega$, where $\wp$ is the Weierstrass function associated with the lattice $\Z\omega_1 + \Z\omega_2$:

$$
     \wp(\omega )=\wp(\omega;\omega_1,\omega_2):=\frac{1}{\omega^{2}}+ \sum_{(\ell_{1},\ell_{2}) \in \Z^{2}\setminus \{(0,0)\}} \left(\frac{1}{(\omega +\ell_{1}\omega_{1}+\ell_{2}\omega_{2})^{2}} -\frac{1}{(\ell_{1}\omega_{1}+\ell_{2}\omega_{2})^{2}}\right).
$$

Then, the field of meromorphic functions on the $\omega$-plane $\Etproj_{\varphi}$ may be identified with the field of meromorphic functions on $\C/(\Z\omega_1 + \Z\omega_2)$, i.e., the field of meromorphic functions on $\C$ that are $(\omega_{1},\omega_{2})$-periodic (or elliptic). Classically, this latter field is equal to $\C(\wp, \wp')$, see \cite{WW}. The map $\Lambda$ induces a bijection from $\C/(\omega_{1}\Z+\omega_{2}\Z)$  to $\Etproj_{\varphi}$ which we will still denote by $\Lambda$. \par

The maps $\iota_{1}$, $\iota_{2}$, and $\sigma$ may be lifted to the $\omega$-plane. We will call them $\iup_{1}$, $\iup_{2}$, and $\widetilde{\sigma}$, respectively. So we have the commutative diagrams 
\begin{equation*}
\xymatrix{
    \Etproj_{\varphi}  \ar@{->}[r]^{\iota_k} & \Etproj_{\varphi} \\
    \C \ar@{->}[u]^\Lambda \ar@{->}[r]_{\iup_k} & \C \ar@{->}[u]_\Lambda 
  }
  \qquad\qquad\qquad
  \xymatrix{
    \Etproj_{\varphi}  \ar@{->}[r]^{\sigma} & \Etproj_{\varphi}  \\
\C \ar@{->}[u]^\Lambda \ar@{->}[r]_{\widetilde{\sigma}} & \C \ar@{->}[u]_\Lambda 
  } 
\end{equation*}
More precisely, following \cite{DuistQRT} (see in particular page 35, Prop~2.5.2,  and Rem.~2.3.8) and \cite[Sec.~2.1]{dreyfus2017differential}, there exists $\omega_{3}\in \C$ such that 
\begin{equation}
\label{eq:expression_group_universal_cover}
     \iup_{1}(\omega)=-\omega, \quad\iup_{2}(\omega)=-\omega+\omega_{3},\quad \text{and} \quad\widetilde{\sigma}(\omega)=\omega+\omega_{3}. 
\end{equation}

We have $\Lambda (\omega)=\Lambda(\omega+\omega_{1})=\Lambda(\omega+\omega_{2})$ for all $\omega\in \C$. This implies that for all $\ell_{1},\ell_{2}\in \Z$, $\iup_{2}$ may be replaced by $\omega\mapsto -\omega+\omega_{3}+\ell_{1}\omega_{1}+\ell_{2}\omega_{2}$. This shows that $\omega_{3}$ is not uniquely defined: it is only defined modulo the lattice $\omega_{1}\Z+\omega_{2}\Z$. \par 

An explicit expression of $\omega_{1},\omega_{2},\omega_{3}$, and $\Lambda$ may be found for instance in \cite[Sec.~2]{dreyfus2017differential}. It is proved that without loss of generality, we may assume that $\omega_{1}$ is purely imaginary, and $\omega_{2},\omega_3$ are positive real numbers, such that $0<\omega_3 <\omega_2$. Note also that the order of $\sigma$ may depends on $t$. It will be finite if and only if $\omega_2/\omega_3 \in \Q$.

\subsection{Analytic continuation}\label{sec22}

Define the domains 
\begin{equation}
\label{eq:three_domains}
     \mathcal{D}_{x} := \Etproj_{\varphi}\cap \{\vert x\vert < 1 \},\quad \mathcal{D}_{y} := \Etproj_{\varphi}\cap \{\vert y\vert < 1 \},\quad \text{and}\quad \mathcal{D}_{x,y}:=\mathcal{D}_{x}\cap \mathcal{D}_{y}. 
\end{equation}   

The coefficients of the series $C_{0-}$ are positive real numbers bounded by $1$. Since $0<t<1$, we find that  $C_{0-}$ converges in $\mathcal{D}_{x}$. Therefore $xf(x)C_{0-}(x)$ converges in $\mathcal{D}_{x}$. Similarly, $D_{\varphi}(y)$ converges in $\mathcal{D}_{y}$ and   $L_{\varphi}(x,y)$ converges in $\mathcal{D}_{x,y}$.  We now follow the strategy of  \cite[Sec.~2]{dreyfus2017differential}. Since the first steps are almost the same, we sketch the method. By \cite[Lem.~2.7]{dreyfus2017differential}, $\mathcal{D}_{x,y}$ is not empty.
 We evaluate \eqref{eq:functional_equation_octant} on 
$\mathcal{D}_{x,y}$  to find, 
$$0=xf(x)C_{0-}(x)+xg(x,y)D_{\varphi}(y)+\frac{t}{2}d_{-1,-1}xC_{0,0}-\frac{xy}{2}.$$

Let us prove that there are at most two points of  $\Etproj_{\varphi}$ with $x$-coordinate equal to $[0:1]$. Indeed, $([0:1],[y_{0}:y_{1}])\in \Etproj_{\varphi}$ implies that $[y_{0}:y_{1}]$ satisfies
$$0=t\sum_{j=0}^2 d^{\varphi}_{-1,j-1} y_0^j y_1^{2-j}=t(d^{\varphi}_{-1,0} y_0 y_1+d^{\varphi}_{-1,-1}  y_1^{2}).$$
Therefore, the only two candidates in $\Etproj_{\varphi}$ are 
$$([0:1],[1:0]) \quad \text{and} \quad ([0:1],[-d^{\varphi}_{-1,-1}:d^{\varphi}_{-1,0}]).  $$

 Let $\mathcal{Z}:= \Etproj_{\varphi}\cap (\{0\}\times \P1(\C))$ denote the finite set of points on $\Etproj_{\varphi}$ with $x$-coordinate that is equal to $[0:1]$.
On $\mathcal{D}_{x,y}\setminus \mathcal{Z}$, we may simplify by $x$:
\begin{equation}\label{eq:funcequaonthecurve2}
0=f(x)C_{0-}(x)+g(x,y)D_{\varphi}(y)+\frac{t}{2}d_{-1,-1}C_{0,0}-\frac{y}{2}.
\end{equation}
Since the right hand side is analytic on $\mathcal{D}_{x,y}$, it is continuous, and \eqref{eq:funcequaonthecurve2} holds on $\mathcal{D}_{x,y}$. 
We continue 
 $g(x,y)D_{\varphi}(y)$ on $\mathcal{D}_{x}$ with the formula  
 $$g(x,y)D_{\varphi}(y)
	=-f(x)C_{0-}(x)-\frac{t}{2}d_{-1,-1}C_{0,0}+\frac{y}{2}. $$ 
	Similarly, we extend $f(x)C_{0-}(x)$ on $\mathcal{D}_{y}$, so both functions have been extended to $\mathcal{D}_{x}\cup \mathcal{D}_{y}$. Let us now see their analytic continuation in the $\omega$-plane. By \cite[Sec.~2]{dreyfus2017differential}, there exists a connected set $\mathcal{O}\subset \C$ such that 
\begin{itemize}
\item $\Lambda (\mathcal{O})=\mathcal{D}_{x}\cup \mathcal{D}_{y}$;
\item $\widetilde{\sigma}^{-1}(\mathcal{O})\cap \mathcal{O}\neq \varnothing$;
\item $\displaystyle \bigcup_{\ell\in \Z}\widetilde{\sigma}^{\ell}(\mathcal{O})=\C$.
\end{itemize}

The monomials $x$ and $y$ are meromorphic on $\Etproj_{\varphi}$. We will denote them by $x(\omega)$ and $y(\omega)$, respectively. They are respectively the functions $\mathfrak{q}_1, \mathfrak{q}_2$ defined in Section \ref{sec21}.  Note that we have an explicit expression for the latter, see \cite[Prop.~2.1]{dreyfus2017differential}.  Simplifying notations, let us write $g_{x,y}(\omega)= g(x(\omega),y(\omega))$.
\par 

The main differences with \cite[Sec.~2]{dreyfus2017differential} begin here, since the factor of $D_{\varphi}(y)$, the polynomial $g(x,y)$, involves $x$ and $y$. The use of  Lemma~\ref{lem3} about the almost independence of $x$ will be a crucial point in the proof of the following theorem. \par 
Recall that the field of meromorphic functions on $\Etproj_{\varphi}$ may be identified with $\C/(\Z\omega_1 + \Z\omega_2)$. The universal cover of $\Etproj_{\varphi}$ may be identified with  $\C$. Meromorphic functions on this universal cover are multivalued functions on $\Etproj_{\varphi}$.
Similarly to the quarter plane case, we have the following functional equations.

\begin{thm}\label{theo2}
The functions $f(x)C_{0-}(x)$ and  $g(x,y)D_{\varphi}(y)$ may be lifted to the universal cover of $\Etproj_{\varphi}$. We will call the continuations $\rx$ and $\ry$ respectively. Seen as functions of $\omega$,
they are meromorphic on $\C$ and satisfy 
\begin{align}
\label{eq:omega_3_per_rx}     \rx(\omega+2\omega_{3}) & =\rx(\omega)+y(\omega+2\omega_{3})-y(\omega+\omega_{3}),\\
     \rx(\omega+\omega_{1})&=  \rx(\omega),\label{eq:omega_1_per_rx}\\
       \ry(\omega+2\omega_{3})   &=\ry(\omega)+\frac{-y(\omega)+2y(\omega+\omega_{3})-y(\omega+2\omega_{3})}{2}, \label{eq:omega_3_per_ry} \\ 
 \ry(\omega+\omega_{1})&=  \ry(\omega).\label{eq:omega_1_per_ry}
\end{align}
\end{thm}

\begin{proof}
  Let $\widetilde{r}_{y}$ be the meromorphic  continuation of $D_{\varphi}(y)$. Consider $x(\omega)$ and $y(\omega)$ as meromorphic functions on $\Etproj_{\varphi}$. 
From \eqref{eq:funcequaonthecurve2}, we deduce that for all $\omega\in \mathcal{O}$,
\begin{equation}\label{eq2}
0= \rx(\omega)+ g_{x,y}(\omega)\widetilde{r}_{y}(\omega)+\frac{td_{-1,-1}C_{0,0}}{2} -\frac{1}{2}y(\omega).
\end{equation}
Recall that $\iup_{1}(\omega)=-\omega$ and $\iota_{1}$ leaves  the first coordinate invariant, so $x(\omega)=x(\iup_{1}(\omega))=x(-\omega)$ and since $r_x$ represent a function of $x$, by the same logic $\rx(\omega)=\rx(-\omega)$. Similarly, $\iup_{2}(\omega)=-\omega+\omega_{3}$,  $y(-\omega)=y(\omega+\omega_{3})$, and $\widetilde{r}_{y}(-\omega)=\widetilde{r}_{y}(\omega+\omega_{3})$. 
\par
By Lemma \ref{lem3}, $g_{x,y}(\widetilde{\sigma}(\omega))=-g_{x,y}(\iup_{1}(\omega))$. We now replace $\omega$ by $\iup_{1}(\omega)$ in both sides. We deduce that for all $\omega\in \widetilde{\sigma}^{-1}(\mathcal{O})\cap \mathcal{O} $,
\begin{equation}\label{eq3}
0= \rx(\omega)-g_{x,y}(\omega+\omega_{3}) \widetilde{r}_{y}(\omega+\omega_{3})+\frac{td_{-1,-1}C_{0,0}}{2} -\frac{y(\omega+\omega_{3})}{2}.
\end{equation}

We now replace $\omega$ by $\iup_{2}(\omega)$  in both sides of \eqref{eq3}, (or equivalently replace $\omega$ by $\widetilde{\sigma}(\omega)$ in both sides of \eqref{eq2}) to obtain for all $\omega\in \widetilde{\sigma}^{-1}(\mathcal{O})\cap \mathcal{O}$,
\begin{equation}\label{eq4}
0= \rx(\omega+\omega_{3})+ g_{x,y}(\omega+\omega_{3})\widetilde{r}_{y}(\omega+\omega_{3})+\frac{td_{-1,-1}C_{0,0}}{2} -\frac{y(\omega+\omega_{3})}{2}.
\end{equation}
By definition, $\ry=g_{x,y}\widetilde{r}_{y}$.  Subtracting  \eqref{eq3} from \eqref{eq2}, we find
\begin{equation}\label{eq5}
r_{y} (\omega+\omega_{3})=  - \ry (\omega) +\frac{y(\omega)-y(\omega+\omega_{3})}{2} . 
\end{equation}

Adding  \eqref{eq3} to \eqref{eq4} we find

\begin{equation}\label{eq6}
\rx (\omega+\omega_{3})=-  \rx (\omega)-td_{-1,-1}C_{0,0}+y(\omega+\omega_{3}) .  
\end{equation}
Recall that $\widetilde{\sigma}^{-1}(\mathcal{O})\cap \mathcal{O}\neq \varnothing$, so that the intersection is an open set with an accumulation point. By the analytic continuation principle we may continue  $\rx$ and $\ry$ on $\widetilde{\sigma}(\mathcal{O})$ with \eqref{eq5} and \eqref{eq6}.  Iterating this strategy, we continue $\rx$ and $\ry$ on $\displaystyle \bigcup_{\ell\in \Z}\widetilde{\sigma}^{\ell}(\mathcal{O})=\C$ and the functions satisfy \eqref{eq5} and \eqref{eq6}. 
Now, we iterate these relations to get \eqref{eq:omega_3_per_ry}
\begin{align*}
\ry (\omega+2\omega_{3})&=- \ry (\omega+\omega_{3}) +\frac{y(\omega+\omega_{3})-y(\omega+2\omega_{3})}{2} \\
&=\ry (\omega)+\frac{-y(\omega)+2y(\omega+\omega_{3})-y(\omega+2\omega_{3})}{2}.
\end{align*}
Similarly, we find  \eqref{eq:omega_3_per_rx}.\par 
The proof of \eqref{eq:omega_1_per_rx} and  \eqref{eq:omega_1_per_ry} is similar to \cite{dreyfus2017differential}: it is based on the fact that $\rx$ and $\ry$ defined on $\mathcal{O}$ are already $\omega_{1}$-periodic. 
\end{proof}

\section{Differential transcendence}
\label{sec:DiffTranscendance}

\subsection{Definition and elementary properties}
The study of algebraic structures of differential fields is the main object of the so-called 
differential algebra.  We refer the interested reader to the founding book of Kolchin, see 
\cite{kolchin1973differential}, for further details.
In what follows, every ring contains $\Q$. In particular every field is of characteristic zero. 
A differential ring $(L,\delta)$ is a ring  $L$ equipped with a derivation $\delta$, that is an additive morphism satisfying the Leibniz rules $\delta(ab)=\delta(a)b+a\delta(b)$ for every $a,b\in L$. If $L$ is additionally a field, we say that it is a differential field. Let $K$ be a differential field and let $L$ be a $K$-algebra. Assume that the derivative of $L$ extends the derivative of $K$. We say that $y\in L$ is {\it $\delta$-algebraic} over $K$ if there exists $n\in \N$ such that $y,\dots, \delta^{n}(y)$ are algebraically dependent over $K$. We say that $y$ is {\it $\delta$-transcendental} otherwise. When no confusions arise, we may also say that $y$ is differentially algebraic or differentially transcendental.\\ \par

We now want to embed the series $C(x,y)$ into a differential field. Since we consider walks with small steps, we have $c_{i,j}(n)=0$ when $|i|,|j|>n$. Therefore, $C(x,y)\in \mathcal{L}:= \Q(x,y)((t))$. Note that $\mathcal{L}$ is a field.  It may be equipped as a differential field with the derivatives $\partial_{x}$, $\partial_{y}$, and $\partial_{t}$. We also have $\Q\subset \Q(x)\subset \mathcal{L}$. Note that since $x$ is $\partial_{x}$-algebraic over $\Q$, we find that $C(x,y)$ is $\partial_{x}$-algebraic over $\Q(x)$ if and only if it is $\partial_{x}$-algebraic over $\Q$. A similar statement holds for $\partial_{y}$. 

\begin{rem}\label{rem1}
Let $f\in \mathcal{L}$. By \cite[Prop.~8, P.~101]{kolchin1973differential}, $f$ is $\partial_{x}$-algebraic over $\C$ if and only if it is $\partial_{x}$-algebraic over $\Q$. A similar statement holds for $\partial_{y}$.
\end{rem}

On the other hand we will have to consider the differential transcendence of $\rx\in \mathcal{M}(\C)$, where $\mathcal{M}(\C)$ denotes the field of meromorphic functions on $\C$.  Note that both fields $\C(\wp,\wp')\subset \mathcal{M}(\C)$ are $\partial_{\omega}$-fields.
As before, $\rx$ is $\partial_{\omega}$-algebraic over $\C(\wp,\wp')$ if and only if it is $\partial_{\omega}$-algebraic over $\C$.
The following lemma shows that to determine the nature of $C(x,y)$, it then suffices to determine the nature of $\rx$. 
\begin{lem}\label{lem2}
The following statements are equivalent.
\begin{enumerate}
\item $C(x,y)$ is $\partial_{x}$-algebraic over $\Q$;
\item $c(x)C_{-0}(x^{-1})$ is $\partial_{x}$-algebraic over $\Q$;
\item $C(x,y)$ is $\partial_{y}$-algebraic over $\Q$;
\item $\widehat{c}(y)C_{0-}(y^{-1})$ is $\partial_{y}$-algebraic over $\Q$;
\item $\rx(\omega)$ is $\partial_{\omega}$-algebraic over $\C$.
\end{enumerate}
\end{lem}

\begin{proof}
Note that $K(x,y)$ and $K(x,y)^{-1}$ are $\partial_{x}$-algebraic over $\Q$. Then 
 $C(x,y)$ is $\partial_{x}$-algebraic over $\Q$ if and only if $K(x,y)C(x,y)$ is  $\partial_{x}$-algebraic over $\Q$. By \eqref{eq1}, $K(x,y)C(x,y)-c(x)C_{-0}(x^{-1})$ is a function of $y$ and therefore is $\partial_{x}$-algebraic over $\Q$. This shows that (i) and (ii) are equivalent. Similarly, we prove that (iii) and (iv) are equivalent. By Assumption \ref{main_hyp}, the weights $d_{i,j}$ are diagonally symmetric so $C(x,y)=C(y,x)$. This shows that (i) and (iii) are equivalent.\par 
Since $f(x)\in \Q[x,t]$ and $\widehat{c}(y)=f(1/y)\in \Q[y,t]$,  (iv) is equivalent to the fact that $f(x)C_{0-}(x)\in\Q(x)((t))$ is  $\partial_{x}$-algebraic over $\Q$. By Remark \ref{rem1}, this is equivalent to the fact that $f(x)C_{0-}(x)$ is $\partial_{x}$-algebraic over $\C$. The map $\Lambda$ locally induces a bijection. 
By construction, we have 
 $$ (fC_{0-}) \circ \Lambda=\rx.$$
Since $\Lambda$ involves the $\partial_{\omega}$-algebraic functions $\wp,\wp'$, we find that $\Lambda$  is $\partial_{\omega}$-algebraic over $\C$. By  \cite[Lem.~6.3 and~6.4]{DHRS}, $f(x)C_{0-}(x)$ is $\partial_{x}$-algebraic over $\C$ if and only if (v) holds. 
\end{proof}

\subsection{On the poles of $\mathbf{a}$}\label{sec32}
By Theorem \ref{theo2}, $\rx$  satisfies an equation of the form 
$$\rx(\widetilde{\sigma}^{2}(\omega))=\rx (\omega)+\widetilde{a}(\omega),$$
where 
$$
 \widetilde{a}(\omega)=
y(\omega+2\omega_{3})-y(\omega+\omega_{3})\in\C(\wp, \wp').
$$
In this subsection, we are going to prove a certain number of technical results concerning the poles of $\widetilde{a}$ and their residues. 
To simplify the computations,  we are going to consider $\widetilde{a}$ as a meromorphic function on $\Etproj_{\varphi}$. In what follows we make the following abuse of notation. By $\sigma^{n}(y)$, we denote the projection on the second coordinate of the map $\sigma^{n}(x,y)$. Let $$\mathbf{a}:=\sigma^{2}(y)-\sigma(y),$$ such that $ \widetilde{a}=\mathbf{a}\circ \Lambda$. \par 
Let us begin by the computation of the poles of $\mathbf{a}$. 
 \begin{lem}\label{lem4}
 The following holds:
 \begin{itemize}
 \item  If $d^{\varphi}_{0,1}=0$, then the poles of $\mathbf{a}$ are double and equal to $$\sigma^{-1}([0:1],[1:0]), \quad \sigma^{-2}([0:1],[1:0]).$$
 \item  If $d^{\varphi}_{0,1}\neq 0$, then
the poles of $\mathbf{a}$ are simple  and their set is included in  
$$\{\sigma^{-1}(P_1), \sigma^{-1}(P_2), \sigma^{-2}(P_1), \sigma^{-2}(P_2)\},$$
where
  $$P_1=([0:1],[1:0])\hbox{ and } P_2=([-d^{\varphi}_{0,1}:d^{\varphi}_{1,1}],[1:0]).$$
 \end{itemize}

 \end{lem}
 
 \begin{proof}
  Let us begin by computing the poles of $y$ and their multiplicities. It suffices to solve  
$$\overline{K}_{\varphi}(x_0,x_1,1,0)=d^{\varphi}_{0,1} x_0 x_1+d^{\varphi}_{1,1}  x_0^{2}=0.$$
We then find that 
\begin{itemize}
\item if $d^{\varphi}_{0,1}=0$, then $y$ has a double pole  $([0:1],[1:0])$.
\item  if $d^{\varphi}_{0,1}\neq 0$, then $y$ has two simple poles,  $([-d^{\varphi}_{0,1}:d^{\varphi}_{1,1}],[1:0])$ and $([0:1],[1:0])$.
\end{itemize}
Note that $P$ is a pole of $y$ if and only if $\sigma^{-n}(P)$ is a pole of $\sigma^{n}(y)$.  By construction of~$\sigma$, $\sigma^{-1}([0:1],[1:0])\neq  \sigma^{-2}([0:1],[1:0])$, concluding the proof.
 \end{proof}
 
 Let us now focus on the dynamics of the poles of $\mathbf{a}$ with respect to $\sigma^{2}$.
We define an equivalence relation on $\Etproj_{\varphi}$ as follows. We say that
 $A,B\in \Etproj_{\varphi}$ satisfy $A\sim B$ if and only if there exists $n\in \Z$, such that $\sigma^{2n}(A)=B$. Then, given poles $A$ and $B$ of $\mathbf{a}$, we want to determine whether $A\sim B$ or not.\par 

 A first partial result is the following. 

\begin{lem}\label{lem5}
Let us assume that $\sigma^{2}$ has infinite order. Let $P\in \Etproj_{\varphi}$. Then, $P\not\sim \sigma(P)$. 
\end{lem}

\begin{proof}
To the contrary, assume the existence of $n\in \Z$ such that $\sigma^{2n}(P)=\sigma(P)$. Then, $\sigma^{2n-1}(P)=P$. 
 Since $\widetilde{\sigma}^{2n-1}(\omega)=\omega+(2n-1)\omega_3$, we find that $\sigma^{2n-1}$ fixes one point if and only if it is the identity. This contradicts the assumption that $\sigma^{2}$ has infinite order.
\end{proof}

Note that the fact that $\sigma^{2}$ has infinite order is crucial in the proof of Lemma~\ref{lem5}.
Let us now focus on the residues of $\widetilde{a}$ when the poles are simple.
 By Lemma \ref{lem4}, we may reduce to the case where $d_{1,1}=d^{\varphi}_{0,1}\neq 0$.
If we see $y$ as a meromorphic function in the $\omega$-plane, we see that, modulo the lattice $\omega_{1}\Z+\omega_{2}\Z$, it has two simple poles corresponding to $P_1$ and $P_2$. Let us see $\Lambda$ as a bijection from $\C/(\omega_{1}\Z+\omega_{2}\Z)$ to $\Etproj_{\varphi}$. Let $[\omega_{y,1}],[\omega_{y,2}]\in \C/(\omega_{1}\Z+\omega_{2}\Z)$ be such that $\Lambda ([\omega_{y,i}])=P_i $. Since $y$ is an elliptic function, the sum of the residues in a fundamental parallelogram is zero. Since the poles of $y$ are simple, their residues are nonzero. This shows that if $0\neq \alpha$ is the residue at $[\omega_{y,1}]$, then the residue at $[\omega_{y,2}]$ is $-\alpha$. Furthermore, it is clear that the residue at $\omega_{0}$ of $f(\omega)\in \mathcal{M}(\C)$ is equal to the residue at  $\omega_{0}-\omega_{3}$ of $f(\widetilde{\sigma}(\omega))$.  
If the  poles $\sigma^{-1}(P_1), \sigma^{-1}(P_2), \sigma^{-2}(P_1), \sigma^{-1}(P_2)$ are distinct, 
this shows that the residue of $\widetilde{a}(\omega)=y(\omega+2\omega_{3})-y(\omega+\omega_{3})$ at $[\omega_{y,1}]-\omega_{3}$ is $-\alpha$, while the residue at $[\omega_{y,2}]-\omega_{3}$ is $\alpha$. Furthermore, the residue of $\widetilde{a}(\omega)$ at $[\omega_{y,1}]-2\omega_{3}$  is $\alpha$, while the residue at $[\omega_{y,2}]-2\omega_{3}$ is $-\alpha$.
To summarize, we have proved that when the  poles $\sigma^{-1}(P_1), \sigma^{-1}(P_2), \sigma^{-2}(P_1), \sigma^{-2}(P_2)$ are distinct the residues of the simple poles of $\widetilde{a}(\omega)$ are equal to 
$$\begin{array}{|l|l|l|l|l|}\hline
\hbox{Pole} & [\omega_{y,1}]-\omega_{3}&[\omega_{y,1}]-2\omega_{3}&[\omega_{y,2}]-\omega_{3}&[\omega_{y,2}]-2\omega_{3}  \\\hline
\hbox{Residue} &-\alpha &\alpha&\alpha&-\alpha \\\hline
\end{array}$$
If the  poles $\sigma^{-1}(P_1), \sigma^{-1}(P_2), \sigma^{-2}(P_1), \sigma^{-2}(P_2)$ are not distinct, we just have to sum up the corresponding residues.

 \subsection{Differential algebraicity}

This subsection is devoted to criteria ensuring that the generating series is differentially algebraic. This is one of our main results.

\begin{thm}\label{thm3}
Suppose that Assumptions \ref{main_hyp} and \ref{main_hyp2} hold. 
Assume further that $d_{1,1}=d^{\varphi}_{0,1}\neq  0$ and that there exists $k\in \Z$, such that for all $0<t<1$, $P_1=\sigma^{2k}(P_2)$. Then, $C(x,y;t)$ is D-algebraic.
\end{thm}
\begin{proof} 
We emphasize that although the $t$ dependencies are not mentioned, almost every function depends on $t$.
By Lemma \ref{lem4}, the poles of $\mathbf{a}$ are simple and their set is included in  $\{\sigma^{-1}(P_1), \sigma^{-1}(P_2), \sigma^{-2}(P_1), \sigma^{-2}(P_2)\}$. Furthermore, $P_1\neq P_2$, and thus $k\neq 0$. Up to interchanging $P_1$ and $P_2$, we may reduce to the case where $k>0$.\par 
We are going to use the following classical properties of $\wp$. It satisfies ${\wp(\omega)=\wp(-\omega)}$,  $\wp'(\omega)=-\wp'(-\omega)$ and for all $\omega_0 \in \C$, the equation $\wp (\omega)=\wp(\omega_0)$ has the solutions ${\pm \omega_0 +\Z \omega_1+ \Z \omega_2}$.  Let us  recall that  $\alpha\neq 0$ is the residue of $y(\omega;t)$ at $[\omega_{y,1}]$ and  $-\alpha$ is the residue of $y(\omega;t)$ at $[\omega_{y,2}]$.  Let $\omega_{y,1}\in \C$  be a representative of $[\omega_{y,1}]$.
  Let $\mathfrak{f}_0(\omega;t)=\frac{1}{\wp(\omega-\omega_{y,1}+3\omega_3/2)-\wp(\omega_3/2)}$. It is an elliptic function with two simple poles at $[\omega_{y,1}]-\omega_3$ and $[\omega_{y,1}]-2\omega_3$ with residues $ \wp'(\omega_3/2)^{-1}$ and $ -\wp'(\omega_3/2)^{-1}$, respectively.
  There exists $c(t)$, such that $\mathfrak{f}(\omega;t)=c(t)\mathfrak{f}_0 (\omega;t)$ is an elliptic function with two simple poles at $[\omega_{y,1}]-\omega_3$ and $[\omega_{y,1}]-2\omega_3$ with respective residues $\alpha$ and $-\alpha$. \par 
In what follows we will say that  a bivariate holomorphic function $h(\omega;t)$ is $(\partial_{\omega},\partial_{t})$-algebraic if it is both $\partial_{\omega}$-algebraic and $\partial_{t}$-algebraic over $\C$.  
We claim that   $\mathfrak{f}(\omega;t)$ is $(\partial_{\omega},\partial_{t})$-algebraic.
By \cite[Lem.~6.10]{bernardi2017counting}, $\omega_1,\omega_2,\omega_3$ are $\partial_{t}$-algebraic over $\C$. By \cite[Prop.~6.7 and Lem.~6.10]{bernardi2017counting}, see also \cite[Prop.~3.5]{dreyfusdiffalg}, $\wp$ is  $(\partial_{\omega},\partial_{t})$-algebraic.  In what follows, we are going to use several times the fact that the differential algebraic functions form a field that is stable under composition, see \cite[Cor.~6.4 and Prop.~6.5]{bernardi2017counting}. With \cite[(2.16)]{dreyfus2017differential}, this shows that $\omega_{y,1}$ and $y(\omega;t)$ are $\partial_t$-algebraic over $\C$, see also \cite[Cor.~3.7]{dreyfusdiffalg}.
Since the product of two differentially algebraic functions is differentially algebraic, we find that 
$(\omega-\omega_{y,1})y(\omega;t)$ is $\partial_t$-algebraic over $\C$. Evaluating the differential algebraic equation at $\omega=\omega_{y,1}$ shows that the residue $\alpha$ is $\partial_t$-algebraic over $\C$. 
 In the same way, we prove that $\mathfrak{f}_{0}(\omega;t)$ and its residues are $\partial_{t}$-algebraic over $\C$. Since the residues of $\mathfrak{f}_0(\omega;t)$ and $\mathfrak{f}(\omega;t)$ are $\partial_{t}$-algebraic over $\C$, we find that $c(t)$, such that $\mathfrak{f}(\omega;t)=c(t)\mathfrak{f}_0(\omega;t)$, is also $\partial_{t}$-algebraic over $\C$. We then deduce that $\mathfrak{f}(\omega;t)$ is $(\partial_{\omega},\partial_{t})$-algebraic.   Let $\mathfrak{g}(\omega;t)=\sum_{\ell=0}^{k-1}\mathfrak{f}(\omega+2\ell \omega_3 ;t)$. From what precedes, $\mathfrak{g}(\omega;t)$ is $(\partial_{\omega},\partial_{t})$-algebraic.
Then,
$\mathfrak{g}(\widetilde{\sigma}^{2}(\omega);t)-\mathfrak{g}(\omega;t)$ has only simple  poles  with corresponding residues
$$\begin{array}{|l|l|l|l|l|}\hline
\hbox{Pole} & [\omega_{y,1}]-(2k+2)\omega_{3}&[\omega_{y,1}]-(2k+1)\omega_{3}&[\omega_{y,1}]-\omega_{3}&[\omega_{y,1}]-2\omega_{3}  \\\hline
\hbox{Residue} &-\alpha &\alpha&-\alpha&\alpha \\\hline
\end{array}$$
On the other hand, recall that the poles of $\widetilde{a}(\omega;t)$ are simple and equal to 
$$\begin{array}{|l|l|l|l|l|}\hline
\hbox{Pole} & [\omega_{y,1}]-\omega_{3}&[\omega_{y,1}]-2\omega_{3}&[\omega_{y,2}]-\omega_{3}&[\omega_{y,2}]-2\omega_{3}  \\\hline
\hbox{Residue} &-\alpha &\alpha&\alpha&-\alpha \\\hline
\end{array}$$
We now use $P_1 =\sigma^{2k}( P_2)$ to deduce that 
$$\begin{array}{|l|l|l|l|l|}\hline
\hbox{Pole} & [\omega_{y,1}]-\omega_{3}&[\omega_{y,1}]-2\omega_{3}&[\omega_{y,1}]-(2k+1)\omega_{3}&[\omega_{y,1}]-(2k+2)\omega_{3} \\\hline
\hbox{Residue} &-\alpha &\alpha&\alpha&-\alpha \\\hline
\end{array}$$
Then, $\mathfrak{g}(\widetilde{\sigma}^{2}(\omega);t)-\mathfrak{g}(\omega;t)$ and $\widetilde{a}(\omega;t)$ have the same poles and residues. This shows that $\mathfrak{g}(\widetilde{\sigma}^{2}(\omega);t)-\mathfrak{g}(\omega;t)-\widetilde{a}(\omega;t)$ is an elliptic function with no poles, and it is therefore constant with respect to $\omega$.
Recall, see Theorem~\ref{theo2}, that $\rx (\widetilde{\sigma}^{2}(\omega);t)-\rx (\omega;t)=\widetilde{a}(\omega;t)$.
Then,
$\mathfrak{g}(\omega;t)-\rx(\omega;t)$ is $\widetilde{\sigma}^{2}$-invariant, i.e., it is $2\omega_3$-periodic. Since $\mathfrak{g}(\omega;t),\rx(\omega;t)$ are $\omega_{1}$-periodic, by Theorem \ref{theo2}, we find that for all $0<t<1$, $\mathfrak{g}(\omega;t)-\rx(\omega;t)=\mathfrak{h}(\omega;t)\in \C(\wp_{1,3},\wp'_{1,3})$, where $\wp_{1,3}$ denotes the Weierstrass function with respect to the periods $\omega_{1},2\omega_{3}$.\par 
  We claim that the poles of $\rx(\omega;t)$  in the $\omega$-plane and their corresponding residues are $\partial_{t}$-algebraic over $\C$. We now use the notations of Section \ref{sec22}. On the set $\mathcal{O}$ with $|x(\omega;t)|<1$, $\rx (\omega;t)$ has no poles. On the set $\mathcal{O}$ with $|y(\omega;t)|<1$, the poles of $\rx (\omega;t)$ are necessarily poles of $g_{x,y}(\omega;t)$. The latter have to be poles of $x(\omega;t)$.  On $\Etproj_{\varphi}$, due to~\eqref{eq:kernel_expanded},
  $$g_{x,y}(\omega;t)=-x(\omega;t)\widehat{a}_{\varphi}(y(\omega;t))-\frac{\widehat{b}_{\varphi}(y(\omega;t))}{2}=x(\omega;t)^{-1}\widehat{c}_{\varphi}(y(\omega;t))+\frac{\widehat{b}_{\varphi}(y(\omega;t))}{2},$$ proving that $g_{x,y}(\omega;t)$ has no poles on $\mathcal{O}$ with $|y(\omega;t)|<1$. Then, $\rx (\omega;t)$ has no poles on $\mathcal{O}$. Recall that the poles of $y(\omega;t)$ are $\partial_{t}$-algebraic and the residues are $\partial_{t}$-algebraic.  By \eqref{eq6} and $\displaystyle \bigcup_{\ell\in \Z}\widetilde{\sigma}^{\ell}(\mathcal{O})=\C$, we deduce that the poles of  $\rx(\omega;t)$  in the $\omega$-plane and their corresponding residues are $\partial_{t}$-algebraic over $\C$. This proves our claim.\par 
  Since the poles of  $\mathfrak{g}(\omega;t)$  in the $\omega$-plane and their corresponding residues are $\partial_{t}$-algebraic over $\C$,  the poles of  $\mathfrak{g}(\omega;t)-\rx(\omega;t)=\mathfrak{h}(\omega;t)$ in the $\omega$-plane and their corresponding residues are $\partial_{t}$-algebraic over $\C$.
For $a,b\in \C$, the function $\frac{1}{\wp_{1,3}(\omega-\frac{a+b}{2})-\wp_{1,3}(\frac{a-b}{2})}$ has simple poles which are located to $a+\Z \omega_1 +\Z 2\omega_3$ and $b+\Z \omega_1 +\Z 2\omega_3$.
With the same reasoning as for $\wp(\omega;t)$, we deduce that the elliptic function $\wp_{1,3}(\omega;t)$ is $(\partial_{\omega},\partial_{t})$-algebraic. 
If $a(t),b(t)$ are $\partial_{t}$-algebraic over $\C$, it follows that $\frac{1}{\wp_{1,3}(\omega-\frac{a(t)+b(t)}{2})-\wp_{1,3}(\frac{a(t)-b(t)}{2})}$,  $\wp_{1,3}(\omega-a(t))$, and $\wp'_{1,3}(\omega-a(t))$ are $(\partial_{\omega},\partial_{t})$-algebraic.
There exists $\mathfrak{h}_0( \omega;t)$, that may be written as a product of functions of the latter form, such that $\mathfrak{h}_0( \omega;t)$ has the same poles as $\mathfrak{h}( \omega;t)$, with the same order. Then, $\mathfrak{h}(\omega;t)/\mathfrak{h}_0(\omega;t)$ is an elliptic function with no poles, it thus does not depend on $\omega$.
Same reasoning as for $y(\omega;t)$ shows that since $\mathfrak{h}_0(\omega;t)$ and its poles are  $\partial_t$-algebraic over $\C$, its residues are $\partial_t$-algebraic over $\C$.
Furthermore, $\mathfrak{h}(\omega;t)$ has  $\partial_{t}$-algebraic residues, which implies that $\mathfrak{h}(\omega;t)/\mathfrak{h}_0(\omega;t)$ is $\partial_{t}$-algebraic over $\C$. Since $\mathfrak{h}_0(\omega;t)$
 is $(\partial_{\omega},\partial_{t})$-algebraic, we find that 
  $\mathfrak{h}(\omega;t)$
 is $(\partial_{\omega},\partial_{t})$-algebraic. This shows that 
$\rx(\omega;t)= \mathfrak{g}(\omega;t)-\mathfrak{h}(\omega;t)$ is $(\partial_{\omega},\partial_{t})$-algebraic.\par 
By Lemma \ref{lem2}, $C(x,y;t)$ is both $\partial_{x}$-algebraic and $\partial_{y}$-algebraic over $\Q$.  It remains to consider the $t$ derivative.
As we have seen in the proof of Lemma \ref{lem2}, $(fC_{0-}) \circ \Lambda=\rx$. We have proved that $y(\omega;t)$ is $\partial_t$-algebraic over $\C$, and in the same way we may prove that $x(\omega;t)$ is $\partial_t$-algebraic over $\C$. Hence,
 the coordinates of the map $\Lambda$ are $\partial_{t}$-algebraic over $\C$, see also \cite[Cor.~3.7]{dreyfusdiffalg}. Then, $C_{0-}(y^{-1})$ is $\partial_{t}$-algebraic over $\C$. Since the weights are diagonally symmetric, $C(x,y;t)=C(y,x;t)$, and $C_{-0}(x^{-1})$ is also $\partial_{t}$-algebraic over $\C$.  By \cite[Prop.~8, P.~101]{kolchin1973differential}, $C_{-0}(x^{-1})$, and 
 $C_{0-}(y^{-1})$ are $\partial_{t}$-algebraic over $\Q$. By \eqref{eq1}, $K(x,y)C(x,y)$ is $\partial_{t}$-algebraic over $\Q$. Then, $C(x,y)$ is $\partial_{t}$-algebraic over $\Q$. This concludes the proof. 
\end{proof}

\begin{ex}\label{ex1}
We may consider  a weighted example. Let $\lambda,\mu\in \Q\cap(0,1)$ such that $\lambda+4\mu=1$. Consider the weighted model with $d_{\pm 1,0}=d_{0,\pm 1}=\mu$, $d_{1,1}=\lambda$, and the other $d_{i,j}$ are $0$ (the condition $\lambda+4\mu=1$ ensures that $\sum_{i,j} d_{i,j}=1$) $$\begin{tikzpicture}[scale=1, baseline=(current bounding box.center)]
\draw[thick,->](0,0)--(-1,0);
\draw[thick,->](0,0)--(1,0);
\draw[thick,->](0,0)--(0,1);
\draw[thick,->](0,0)--(0,-1);
\draw[thick,->](0,0)--(1,1);
\put(-10,20){{$\mu$}}
\put(25,15){{$\lambda$}}
\put(-10,-25){{$\mu$}}
\put(25,-10){{$\mu$}}
\put(-25,-10){{$\mu$}}
\end{tikzpicture}$$
Let $C_{\lambda,\mu}(x,y;t)$ be the corresponding generating series.
 After the change of variable $\varphi$ we obtain the following step set
$$\begin{tikzpicture}[scale=1, baseline=(current bounding box.center)]
\draw[thick,->](0,0)--(-1,0);
\draw[thick,->](0,0)--(1,0);
\draw[thick,->](0,0)--(0,1);
\draw[thick,->](0,0)--(-1,-1);
\draw[thick,->](0,0)--(1,1);
\put(-10,20){{$\lambda$}}
\put(25,15){{$\mu$}}
\put(-15,-25){{$\mu$}}
\put(25,-10){{$\mu$}}
\put(-25,-10){{$\mu$}}
\end{tikzpicture}$$
We have $$P_1=([0:1],[1:0])\hbox{ and } P_2=([-\lambda:\mu],[1:0]).$$ 
As in \cite[Sec.~6]{DHRS}, the following orbit holds: 
$$
\xymatrix{
    & P_{1} =([0:1],[1:0])  \ar[ld]^{\iota_{1}} \ar[d]_{\sigma} \\
    ([0:1],[-1:1]) \ar[r]_{\iota_{2}} &([1:0],[-1:1]) \ar[ld]^{\iota_{1}} \ar[d]_{\sigma}\\
    ([1:0],[0:1]) \ar[r]_{\iota_{2}}  & ([1:0],[0:1]) \ar[ld]^{\iota_{1}} \ar[d]_{\sigma} \\
   ([1:0],[-1:1])  \ar[r]_{\iota_{2}}  & ([0:1],[-1:1]) \ar[ld]^{\iota_{1}} \ar[d]_{\sigma}  \\
  ([0:1],[1:0]) \ar[r]_{\iota_{2}}  &  P_{2} = ([-\lambda:\mu],[1:0]). 
  }
  $$
Therefore,  $\sigma^{4}(P_1)=P_2$, for all $0<t<1$.
By Theorem \ref{thm3}, we then obtain that $C_{\lambda,\mu}(x,y;t)$ is $D$-algebraic.
\end{ex}

 \subsection{Differential transcendence}

In this section, we want to prove results on differential transcendence. Thanks to Lemma \ref{lem2}, it suffices to prove that $\rx(\omega)$ is $\partial_{\omega}$-transcendental over $\C$ for a fixed transcendental value of $t\in (0,1)$. \par 
From now on, let us fix $t\in (0,1)$ transcendental  and let us
 make the following assumption:

\begin{enumerate}[label=\textcolor{red}{\textbf{(A3)}},ref={\rm (A3)}]
     \item\label{main_hyp3} The automorphism $\sigma$ has infinite order.
\end{enumerate}

 In the paper \cite{DHRS},  differential transcendence criteria for the generating series of walks in the quarter plane are given. Similarly to this paper, the generating series admits a continuation $f_{x}(\omega)$ that is meromorphic on $\C$ and satisfies a functional equation of the form $f_{x} (\widetilde{\sigma}(\omega))=f_{x} (\omega)+b(\omega)$, with $b(\omega)\in \C(\wp(\omega),\wp'(\omega))$. Under the assumption that $\widetilde{\sigma}$ has infinite order, the difference Galois theory of \cite{HS} is used to provide  differential transcendence criteria based on $b(\omega)$.
Note that by  Assumption \ref{main_hyp3} $\widetilde{\sigma}^2$ has infinite order. More generally,  the criteria of \cite{DHRS} stay valid in this more general context to provide the differential transcendence of $y$, a meromorphic function on $\C$, that is additionally a solution of a functional equation of the form $y (\widetilde{\sigma}^{2}(\omega))=y(\omega)+\tilde{b}(\omega)$, with $\tilde{b}(\omega)\in \C(\wp(\omega),\wp'(\omega))$ and $\widetilde{\sigma}^{2}$ is of infinite order.
By Theorem \ref{theo2}, 
$\rx (\omega)$ is meromorphic on $\C$ and satisfies an equation of the form  $\rx (\widetilde{\sigma}^{2}(\omega))=\rx (\omega)+\widetilde{a}(\omega)$, where $\widetilde{a}(\omega)\in \C(\wp(\omega),\wp'(\omega))$. We are then in the framework where the criteria of \cite{DHRS} apply.

\begin{prop}[Cor.~3.7, \cite{DHRS}]\label{prop4}
Assume that $\mathbf{a}$ has a pole $P$ of order $\geq m$, such that none of the $\sigma^{2\ell}(P)$, with $\ell\in \Z^{*}$, is a pole of order $\geq m$ of $\mathbf{a}$. Then, $\rx$ is $\partial_{\omega}$-transcendental over $\C$.
\end{prop}

 Unfortunately, it may happen that Proposition \ref{prop4} is not sufficient to prove the differential transcendence of $\rx$.  We now state an equivalent condition to the  $\partial_{\omega}$-transcendence of $\rx$ that involves the  residues of $\widetilde{a}$.
Note that  $\widetilde{a}$ is meromorphic on $\C$ and is $(\omega_{1},\omega_{2})$-periodic. This shows that there exist $\omega_{\widetilde{a},1},\dots,\omega_{\widetilde{a},k}\in \C$ such that the set of poles of $\widetilde{a}$ is included in $\cup_{\ell=1}^{k}\omega_{\widetilde{a},\ell}+\omega_{1}\Z+\omega_{2}\Z$.
Recall, see Section \ref{sec21}, that $\omega_1$ is purely imaginary, $\omega_2,\omega_3$ are positive real numbers, and 
 $\sigma$ has infinite order if and only if $\omega_2/\omega_3 \notin \Q$.
 By \ref{main_hyp3}, $\widetilde{\sigma}^{2}$ has infinite order, so $\omega_2/\omega_3 \notin \Q$ and we find that for all $\omega_{0}\in \C$ and for all $\ell=1,\dots,k$, $(\omega_0+2\Z\omega_{3})\cap (\omega_{\widetilde{a},\ell}+\omega_{1}\Z+\omega_{2}\Z)$ has cardinality at most one. Then,  $\widetilde{a}$ has a finite number of poles in  $\omega_0+2\Z\omega_{3}$. Note that as $\widetilde{\sigma}^{2}(\omega)=\omega+2\omega_3$, $\omega_0+2\omega_{3}\Z$ is the orbit of $\omega_{0}$.

\begin{prop}[Prop.~B.2, \cite{DHRS}]\label{prop5}
The function $\rx$ is $\partial_{\omega}$-algebraic over $\C$ if and only if for all $\omega_{0}\in \C$, the following function is analytic at $\omega_0$
$$\sum_{i=1}^{t} \widetilde{a}(\omega+2n_{i}\omega_{3}),$$
where $n_{i}$ are the integers such that $\omega_0+2n_{i}\omega_{3}$  are poles of  $\widetilde{a}$.
\end{prop}

When all the poles of $\widetilde{a}$ are simple, the condition that $\sum_{i=1}^{t} \widetilde{a}(\omega+2n_{i}\omega_{3})$ is analytic at  $\omega_0$ is equivalent to the fact that the sum of the residues of $\widetilde{a}$ at $\omega+2n_{i}\omega_{3}$ is zero.

We are now ready to state and prove one of our main results.
\begin{thm}\label{thm1}
Let us fix a transcendental value of $t$ in $(0,1)$. Assume that Assumptions \ref{main_hyp}, \ref{main_hyp2}, and \ref{main_hyp3} (for this transcendental value of $t$) hold.
Assume further that $d_{1,1}=d^{\varphi}_{0,1}= 0$. Then, $C(x,y;t)$ is $D$-transcendental.  More precisely, for the fixed value of $t$, $C(x,y)$ is $\partial_{x}$-transcendental over $\Q$ and  $\partial_{y}$-transcendental over $\Q$.
\end{thm}

\begin{rem}
By Remark \ref{rem1}, for the fixed value of $t$, $C(x,y)$ is $\partial_{x}$-transcendental over $\C$ and  $\partial_{y}$-transcendental over $\C$.
\end{rem}
\begin{proof}
By Lemma \ref{lem4}, $\mathbf{a}$ has two double poles, $\sigma^{-1}([0:1],[1:0])$, $\sigma^{-2}([0:1],[1:0])$.  By Lemma \ref{lem5} they do not belong to the same $\sigma^{2}$-orbits.  By Proposition \ref{prop4} and Lemma~\ref{lem2}, $C(x,y)$ is $\partial_{x}$-transcendental over $\Q$ and  $\partial_{y}$-transcendental over $\Q$.
\end{proof}

\begin{ex}
Consider the weighted model with $d_{\pm 1,0}=d_{0,\pm 1}=1/6$, $d_{-1,-1}=1/3$, and the other $d_{i,j}$ are $0$ $$\begin{tikzpicture}[scale=1, baseline=(current bounding box.center)]
\draw[thick,->](0,0)--(-1,0);
\draw[thick,->](0,0)--(1,0);
\draw[thick,->](0,0)--(0,1);
\draw[thick,->](0,0)--(0,-1);
\draw[thick,->](0,0)--(-1,-1);
\put(5,20){{$1/6$}}
\put(-45,-25){{$1/3$}}
\put(5,-25){{$1/6$}}
\put(25,-10){{$1/6$}}
\put(-45,-10){{$1/6$}}
\end{tikzpicture}$$
Assumptions \ref{main_hyp}, \ref{main_hyp2} obviously hold and we have $d_{1,1}=d^{\varphi}_{0,1}= 0$. Using the fixed point method of \cite[Sec.~3]{BMM}, we can prove that for $t$ transcendental, Assumption \ref{main_hyp3} holds. 
 Then, $C(x,y;t)$ is $D$-transcendental.
\end{ex}

Recall, see Lemma \ref{lem5}, that $P_i\not\sim \sigma(P_i)$.
From the discussion on the residues in Section~\ref{sec32}, when the poles of $y$ are simple, the conclusion of Proposition~\ref{prop5} holds positively if and only if
$P_1 \sim  P_2$ (and then $\sigma(P_1) \sim \sigma( P_2))$. 
We have proved:

\begin{prop}\label{prop6}
Let us fix a transcendental value of $t$ in $(0,1)$. Assume that Assumptions \ref{main_hyp}, \ref{main_hyp2}, and \ref{main_hyp3} (for this transcendental value of $t$) hold.
Assume further that $d_{1,1}=d^{\varphi}_{0,1}\neq  0$. Then, $C(x,y)$ is $\partial_{x}$-algebraic over $\Q$ (resp.~$\partial_{y}$-algebraic over $\Q$) if and only if 
$$P_1 \sim P_2.$$
\end{prop}

\begin{ex}
Consider the weighted model with $d_{ 1,0}=d_{0, 1}=d_{1,1}=1/5$, $d_{-1,-1}=2/5$, and the other $d_{i,j}$ are $0$ 
$$\begin{tikzpicture}[scale=1, baseline=(current bounding box.center)]
\draw[thick,->](0,0)--(1,0);
\draw[thick,->](0,0)--(0,1);
\draw[thick,->](0,0)--(-1,-1);
\draw[thick,->](0,0)--(1,1);
\put(-20,20){{$1/5$}}
\put(25,15){{$1/5$}}
\put(-15,-25){{$2/5$}}
\put(25,-10){{$1/5$}}
\end{tikzpicture}$$
Assumptions \ref{main_hyp}, \ref{main_hyp2} obviously hold and we have $d_{1,1}=d^{\varphi}_{0,1}\neq 0$. Using the fixed point method of \cite[Sec.~3]{BMM}, we can prove that for $t$ transcendental, Assumption \ref{main_hyp3} holds.  We have $$P_1=([0:1],[1:0])\hbox{ and } P_2=([-1:1],[1:0]).$$
A straightforward computation shows that $\iota_1 (P_1)=P_1$ and $\sigma(P_1)=P_2$. By Lemma \ref{lem5}, $P_1\not\sim \sigma(P_1)=P_2$. Then, $C(x,y;t)$ is $D$-transcendental.
\end{ex}
\subsection{Unweighted cases}
Let us now look at the
unweighted diagonally symmetric models with infinite group:

	\begin{center}\begin{tabular}{cccc}
		$\ \diagr{E,NE,N,SW}\ $
		\qquad&\qquad 
		$\ \diagr{W,NE,S,SW}\ $ 
		\qquad&\qquad
		$\ \diagr{E,W,N,S,SW}\ $
		\qquad&\qquad
		$\ \diagr{E,W,N,S,NE}\ $ 
	\end{tabular}\end{center}
	
After performing the change of variable $\varphi$, we obtain models corresponding to the following steps:

	\begin{center}\begin{tabular}{cccc}
		$\ \diagr{W,N,S,NE}\ $
		\qquad&\qquad 
		$\ \diagr{N,E,S,SW}\ $ 
		\qquad&\qquad
		$\ \diagr{E,W,SW,S,NE}\ $
		\qquad&\qquad
		$\ \diagr{E,W,N,SW,NE}\ $ 
	\end{tabular}\end{center}

Let us now apply the results of the previous sections in order to determine the nature of the generating series.
	\begin{thm}\label{thm2}
The following holds:
\begin{itemize}
\item Assume that the model of the walk is the following: 
$$ \begin{tikzpicture}[scale=.5, baseline=(current bounding box.center)]
\draw[thick,->](0,0)--(0,1);
\draw[thick,->](0,0)--(0,-1);
\draw[thick,->](0,0)--(1,0);
\draw[thick,->](0,0)--(-1,0);
\draw[thick,->](0,0)--(1,1);
\end{tikzpicture}$$
Then,
$C(x,y;t)$ is D-algebraic.
\item Assume that the model of the walk is one of the following:
$$\begin{tikzpicture}[scale=.5, baseline=(current bounding box.center)]
\draw[thick,->](0,0)--(0,1);
\draw[thick,->](0,0)--(1,1);
\draw[thick,->](0,0)--(1,0);
\draw[thick,->](0,0)--(-1,-1);
\end{tikzpicture}\quad  \quad \begin{tikzpicture}[scale=.5, baseline=(current bounding box.center)]
\draw[thick,->](0,0)--(0,-1);
\draw[thick,->](0,0)--(1,1);
\draw[thick,->](0,0)--(-1,0);
\draw[thick,->](0,0)--(-1,-1);
\end{tikzpicture}\quad \quad  \begin{tikzpicture}[scale=.5, baseline=(current bounding box.center)]
\draw[thick,->](0,0)--(0,1);
\draw[thick,->](0,0)--(0,-1);
\draw[thick,->](0,0)--(1,0);
\draw[thick,->](0,0)--(-1,0);
\draw[thick,->](0,0)--(-1,-1);
\end{tikzpicture}$$
Then, $C(x,y;t)$ is D-transcendental. More precisely, there exists $t\in(0,1)$ such that $C(x,y)$ is $\partial_{x}$-transcendental over $\Q$. The same holds for $\partial_{y}$. 
\end{itemize}
	\end{thm}

\begin{proof}
The first statement is a particular case of  Example \ref{ex1}. Let us now prove the second statement. 
We want to apply Theorem \ref{thm1} and Proposition \ref{prop6}, so we have to check that the three assumptions are satisfied.
Assumptions \ref{main_hyp} and \ref{main_hyp2} are clearly verified in the  models under consideration. Let us prove that there exists a transcendental number $t\in (0,1)$ such that
Assumption \ref{main_hyp3} holds.
By \cite[Prop.~2.6]{DHRS}, the set of $t\in (0,1)$ such that the group specialized at $t$ is finite is denumerable. Since the transcendental numbers in $(0,1)$ are not denumerable, this implies the existence of a transcendental number $t\in (0,1)$ such that \ref{main_hyp3} is verified (for this value of $t$).\par
In the third case we use Theorem \ref{thm1} to deduce the differential transcendence because $d^{\varphi}_{0,1}= 0$. Let us assume that  $d^{\varphi}_{0,1}\neq 0$. We want to apply Proposition~\ref{prop6}. Fortunately, in this situation, the $\sigma$-orbits (and therefore the $\sigma^{2}$-orbits) have been computed in \cite[Sec.~6]{DHRS} (more precisely, see the models wIIC.1, wIIB.1, and wIIC.2)
 $$\sigma^{-1}(P_1)\sim \sigma^{-2}(P_2), \quad \sigma^{-2}(P_1)\sim\sigma^{-1}(P_2) .$$
	\end{proof}
	
\bibliographystyle{alpha}
 \bibliography{bibdata}

\newcommand{\etalchar}[1]{$^{#1}$}
\begin{thebibliography}{BCVH{\etalchar{+}}17}

\bibitem[BBMM20]{bostan2018counting}
A.~Bostan, M.~Bousquet-M\'elou, and S.~Melczer.
\newblock On walks with large steps in an orthant.
\newblock {\em Journal of the European Mathematical Society (JEMS)}, 2020.

\bibitem[BBMR17]{bernardi2017counting}
O.~Bernardi, M.~Bousquet-M{\'e}lou, and K.~Raschel.
\newblock Counting quadrant walks via {T}utte's invariant method.
\newblock {\em arXiv preprint arXiv:1708.08215}, 2017.

\bibitem[BCVH{\etalchar{+}}17]{BoChVHKaPe-17}
A.~Bostan, F~Chyzak, M.~Van~Hoeij, M.~Kauers, and L.~Pech.
\newblock Hypergeometric expressions for generating functions of walks with
  small steps in the quarter plane.
\newblock {\em European J. Combin.}, 61:242--275, 2017.

\bibitem[BF02]{BaFl-02}
C.~Banderier and P.~Flajolet.
\newblock Basic analytic combinatorics of directed lattice paths.
\newblock {\em Theoret. Comput. Sci.}, 281(1-2):37--80, 2002.
\newblock Selected papers in honour of Maurice Nivat.

\bibitem[BK10]{BostanKauersTheCompleteGenerating}
A.~Bostan and M.~Kauers.
\newblock The complete generating function for {G}essel walks is algebraic.
\newblock {\em Proceedings of the American Mathematical Society},
  138(9):3063--3078, 2010.

\bibitem[BM16]{BM-16}
M.~Bousquet-M{\'e}lou.
\newblock Square lattice walks avoiding a quadrant.
\newblock {\em Journal of Combinatorial Theory, Series A}, 144:37--79, 2016.

\bibitem[BMM10]{BMM}
M.~Bousquet-M\'elou and M.~Mishna.
\newblock Walks with small steps in the quarter plane.
\newblock In {\em Algorithmic probability and combinatorics}, volume 520 of
  {\em Contemp. Math.}, pages 1--39. Amer. Math. Soc., Providence, RI, 2010.

\bibitem[BMW20]{BMW-20}
M.~Bousquet-M\'elou and M.~Wallner.
\newblock {More Models of Walks Avoiding a Quadrant}.
\newblock {\em 31st International Conference on Probabilistic, Combinatorial
  and Asymptotic Methods for the Analysis of Algorithms (AofA 2020)},
  159:8:1--8:14, 2020.

\bibitem[Bud20]{Bu-20}
T.~Budd.
\newblock Winding of simple walks on the square lattice.
\newblock {\em J. Combin. Theory Ser. A}, 172:105191, 59, 2020.

\bibitem[DH19]{dreyfus2019length}
T.~Dreyfus and C.~Hardouin.
\newblock Length derivative of the generating series of walks confined in the
  quarter plane.
\newblock {\em arXiv preprint arXiv:1902.10558}, 2019.

\bibitem[DHRS18]{DHRS}
T.~Dreyfus, C.~Hardouin, J.~Roques, and M.~Singer.
\newblock On the nature of the generating series of walks in the quarter plane.
\newblock {\em Inventiones mathematicae}, pages 139--203, 2018.

\bibitem[DHRS20]{dreyfus2020walks}
T.~Dreyfus, C.~Hardouin, J.~Roques, and M.~Singer.
\newblock Walks in the quarter plane: Genus zero case.
\newblock {\em Journal of Combinatorial Theory, Series A}, 174:105251, 2020.

\bibitem[DHRS21]{DHRS20}
T.~Dreyfus, C.~Hardouin, J.~Roques, and M.~Singer.
\newblock On the kernel curves associated with walks in the quarter plane.
\newblock {\em Springer Proceedings in Mathematics and Statistics}, 2021.

\bibitem[DR19]{dreyfus2017differential}
T.~Dreyfus and K.~Raschel.
\newblock Differential transcendence \& algebraicity criteria for the series
  counting weighted quadrant walks.
\newblock {\em Publications Math{\'e}matiques de Besan{\c{c}}on}, (1):41--80,
  2019.

\bibitem[Dre21]{dreyfusdiffalg}
T.~Dreyfus.
\newblock Differential algebraic generating series of weighted walks in the
  quarter plane.
\newblock {\em arXiv 2104.05505}, 2021.

\bibitem[Dui10]{DuistQRT}
J.~Duistermaat.
\newblock {\em Discrete Integrable Systems: Qrt Maps and Elliptic Surfaces},
  volume 304 of {\em Springer Monographs in Mathematics}.
\newblock Springer-Verlag, New York, 2010.

\bibitem[DW15]{DeWa-15}
D.~Denisov and V.~Wachtel.
\newblock Random walks in cones.
\newblock {\em Ann. Probab.}, 43(3):992--1044, 2015.

\bibitem[FIM17]{FIM}
G.~Fayolle, R.~Iasnogorodski, and V.~Malyshev.
\newblock {\em Random walks in the quarter plane}, volume~40 of {\em
  Probability Theory and Stochastic Modelling}.
\newblock Springer, Cham, second edition, 2017.
\newblock Algebraic methods, boundary value problems, applications to queueing
  systems and analytic combinatorics.

\bibitem[FR10]{fayolle2010holonomy}
G.~Fayolle and K.~Raschel.
\newblock On the holonomy or algebraicity of generating functions counting
  lattice walks in the quarter-plane.
\newblock {\em Markov Process. Related Fields}, 16(3):485--496, 2010.

\bibitem[FR11]{fayolleRaschel}
G.~Fayolle and K.~Raschel.
\newblock Random walks in the quarter-plane with zero drift: an explicit
  criterion for the finiteness of the associated group.
\newblock {\em Markov Process. Related Fields}, 17(4):619--636, 2011.

\bibitem[FR12]{FaRa-12}
G.~Fayolle and K.~Raschel.
\newblock Some exact asymptotics in the counting of walks in the quarter plane.
\newblock In {\em 23rd {I}ntern. {M}eeting on {P}robabilistic, {C}ombinatorial,
  and {A}symptotic {M}ethods for the {A}nalysis of {A}lgorithms ({A}of{A}'12)},
  Discrete Math. Theor. Comput. Sci. Proc., AQ, pages 109--124. Assoc. Discrete
  Math. Theor. Comput. Sci., Nancy, 2012.

\bibitem[HS08]{HS}
C.~Hardouin and M.~Singer.
\newblock Differential {G}alois theory of linear difference equations.
\newblock {\em Math. Ann.}, 342(2):333--377, 2008.

\bibitem[Kol73]{kolchin1973differential}
E.~R. Kolchin.
\newblock {\em Differential algebra \& algebraic groups}.
\newblock Academic press, 1973.

\bibitem[KR11]{KuRa-11}
I.~Kurkova and K.~Raschel.
\newblock Explicit expression for the generating function counting {G}essel's
  walks.
\newblock {\em Adv. in Appl. Math.}, 47(3):414--433, 2011.

\bibitem[KR12]{kurkova2012functions}
I.~Kurkova and K.~Raschel.
\newblock On the functions counting walks with small steps in the quarter
  plane.
\newblock {\em Publications math{\'e}matiques de l'IH{\'E}S}, 116(1):69--114,
  2012.

\bibitem[KR15]{KuRa-15}
I.~Kurkova and K.~Raschel.
\newblock New steps in walks with small steps in the quarter plane: series
  expressions for the generating functions.
\newblock {\em Ann. Comb.}, 19(3):461--511, 2015.

\bibitem[KY15]{KauersYatchak}
M.~Kauers and R.~Yatchak.
\newblock Walks in the quarter plane with multiple steps.
\newblock In {\em Proceedings of {FPSAC} 2015}, Discrete Math. Theor. Comput.
  Sci. Proc., pages 25--36. Assoc. Discrete Math. Theor. Comput. Sci., Nancy,
  2015.

\bibitem[MM14]{MeMi-14}
S.~Melczer and M.~Mishna.
\newblock Singularity analysis via the iterated kernel method.
\newblock {\em Combin. Probab. Comput.}, 23(5):861--888, 2014.

\bibitem[MM16]{MeMi-16}
S.~Melczer and M.~Mishna.
\newblock Asymptotic lattice path enumeration using diagonals.
\newblock {\em Algorithmica}, 75(4):782--811, 2016.

\bibitem[MR09]{MiRe-09}
M.~Mishna and A.~Rechnitzer.
\newblock Two non-holonomic lattice walks in the quarter plane.
\newblock {\em Theoret. Comput. Sci.}, 410(38-40):3616--3630, 2009.

\bibitem[Mus19]{mustapha2019non}
S.~Mustapha.
\newblock Non-{D}-finite walks in a three-quadrant cone.
\newblock {\em Annals of Combinatorics}, 23(1):143--158, 2019.

\bibitem[MW20]{melczer2015asymptotics}
S.~Melczer and M.~Wilson.
\newblock Asymptotics of lattice walks via analytic combinatorics in several
  variables.
\newblock {\em Discrete Mathematics and Theoretical Computer Science}, 2020.

\bibitem[Ras12]{RaschelJEMS}
K.~Raschel.
\newblock Counting walks in a quadrant: a unified approach via boundary value
  problems.
\newblock {\em Journal of the European Mathematical Society (JEMS)},
  14(3):749--777, 2012.

\bibitem[RT19]{RaTr-18}
K.~Raschel and A.~Trotignon.
\newblock On walks avoiding a quadrant.
\newblock {\em Electronic Journal of Combinatorics}, 26:1--34, 2019.

\bibitem[Tro19]{Tr-19}
A.~Trotignon.
\newblock Discrete harmonic functions in the three-quarter plane.
\newblock {\em arXiv}, 1906.08082:1--26, 2019.

\bibitem[WW96]{WW}
E.~Whittaker and G.~Watson.
\newblock {\em A course of modern analysis}.
\newblock Cambridge Mathematical Library. Cambridge University Press,
  Cambridge, 1996.
\newblock An introduction to the general theory of infinite processes and of
  analytic functions; with an account of the principal transcendental
  functions, Reprint of the fourth (1927) edition.

\end{thebibliography}

\end{document}